\documentclass[11pt,leqno]{article}

\usepackage{fontenc}
\usepackage{lmodern}      

\usepackage{amsmath,amssymb,amsthm,mathtools}

\DeclareMathAlphabet{\mathfrak}{U}{euf}{m}{n}

\numberwithin{equation}{section}

\makeatletter
\renewcommand{\tagform@}[1]{\maketag@@@{[#1]\ignorespaces}}
\makeatother



\usepackage{epigraph}

\usepackage{enumitem}

\usepackage{refcount}     

\usepackage{eurosym}

\usepackage{mathrsfs}

\usepackage{textcomp}

\DeclareRobustCommand{\gbp}{%
  \ifmmode
    \mathbin{\text{\pounds}}%
  \else
    \pounds
  \fi
}

\DeclareRobustCommand{\lb}{%
  \ifmmode
    \mathbin{\text{\pounds}}%
  \else
    \pounds
  \fi
}

\usepackage{bm}

\usepackage{tikz}

\usepackage{titlesec}

\newcommand{\Bendcode}[2][]{%
  \mspace{1\medmuskip}%
  \vphantom{#2}%
  \begin{tikzpicture}[baseline=(M.south)]
    \node[inner ysep=0pt,inner xsep=4pt](M){\smash[b]{$#2\mathstrut$}};
    \draw[rounded corners=.5mm,#1]
      ([xshift=-1mm]M.south east)--(M.south east)--++(0,.14);
    \draw[rounded corners=.5mm,#1]
      ([xshift=1mm]M.south west)--(M.south west)--++(0,.14);
  \end{tikzpicture}%
  \mspace{1\medmuskip}%
}

\makeatletter
\newtheoremstyle{fabthm}%
  {1.5ex plus 0.5ex}
  {1.5ex plus 0.5ex}
  {\rmfamily}
  {}
  {}
  {}
  {0.5em}
  {%
    \thmname{\textsc{#1}}%
    \if\relax\detokenize{#2}\relax
    \else
      \ \thmnumber{\textsc{#2}}%
    \fi
    \if\relax\detokenize{#3}\relax
    \else
      \ #3%
    \fi
  }
\makeatother

\theoremstyle{fabthm}

\newtheorem{thm}{Theorem}[section]

\newtheorem{axiom}[thm]{Axiom}

\newtheorem{conv}[thm]{Convention}
\newtheorem{cor}[thm]{Corollary}
\newtheorem{dfn}[thm]{Definition}

\newtheorem{fact}[thm]{Fact}

\newtheorem{lem}[thm]{Lemma}

\newtheorem{postulate}[thm]{Postulate}

\newtheorem{proposition}[thm]{Proposition}

\newtheorem{remark}[thm]{Remark}

\newtheorem{substitutiontheorem}[thm]{The substitution Theorem}

\newcommand{\VisionLeft}{\mathopen{\langle\!\langle}}
\newcommand{\VisionRight}{\mathclose{\rangle\!\rangle}}
\newcommand{\DefEqual}{\mathrel{:=}}
\DeclareRobustCommand{\Visible}{\mathop{\bm{!}}}
\DeclareRobustCommand{\Invisible}{\mathop{\text{\bfseries\textexclamdown}}}
\newcommand{\TruthPred}{\operatorname{\ensuremath{\mathsf{Tr}}}}
\DeclareRobustCommand{\Setpredicate}{\operatorname{\mathsf{Set}}}

\newcommand{\Anschauung}[2]{\VisionLeft #1 : #2 \VisionRight}
\newcommand{\Proposisjon}[1]{\VisionLeft #1 \VisionRight}

\newcommand{\EndDef}{\par\hfill $\Diamond$}

\usepackage[backend=biber,isbn=false,style=authoryear]{biblatex}

\renewbibmacro{in:}{}

\usepackage[hidelinks]{hyperref}
\usepackage{cleveref}

\title{Visibilism and Visibility Theory:\\[0.3em]
Sets, Truth, Paradox, Provability, Visibility, Visions}

\author{Frode Alfson Bjørdal} 
\date{}

\begin{document}

\maketitle

\thispagestyle{empty}

\medskip

\bigskip

\bigskip

\epigraph{
The paradox is not\\ 
that inconsistency is visible,\\
but that invisibility is derivable.
}{}
\setlength\epigraphrule{0pt}

\newpage

\tableofcontents

\newpage

\section{Introduction}

\setlength\epigraphrule{0pt}

\epigraph{Theorems are at least provable.}{}

\noindent Visibility Theory develops out of earlier librationist work, as \parencite{Bjordal2012, Bjordal2025arxiv}, but replaces their \emph{contrasistent} approaches with a visibility-theoretic framework grounded in visions, visibility, ordinal permanence and Set. Visibility inherits asymmetries and tensions reminiscent of librationist truth theory. While important structural techniques from librationism, as manifestation constructions, remain available, \textbf{VT} recovers classical truth and set theory within the domain of \emph{propositions}.

The librationist accounts were extensions of Cantini's system \textbf{LES}, of \parencite[400]{Cantini1996}. Where those frameworks employed truth and sorts, \textbf{VT} employs visibility and visions.
The resulting framework retains structural techniques 
from the contrasistent librationism, but its treatment of semantic paradoxes by means of visibility, in a framework grounded in Set and ordinal permanence, does not involve any moves which deviate from fully classical reasoning. 

\newpage

\section{Preliminary overview of some main notions}

\subsection{Introducing Bend codes}

The basic formal language of Visibility Theory is Polish. As explained in \S\ref{Bendcodes andvisions}, Bend codes are recursively assigned to primitive symbol strings by means of the well-ordering $(S,\prec)$ of Definition~\ref{stringset}:~p.~\pageref{stringset}. The coding is defined internally within the formal framework itself, rather than through an external arithmetization as with Gödel numbering.

If $X$ is a formal expression, then so is its Bend code $\Bendcode{X}$; and the latter has, on account of Definition~\ref{bendcodesprimitivestrings}:~p.~\pageref{bendcodesprimitivestrings}, a corresponding \emph{vision} as denotatum.

\subsection{Entering visions, identity and Set}

The notion of \emph{vision} is fundamental in Visibility Theory, and $\mathsf{V}$ is a variable binding operator such that $\mathsf{V}yA$ is the  \emph{vision} of those $y$ that are $A$. By Convention \ref{prefixtoinfixdefinition}:~p.~\pageref{prefixtoinfixdefinition} on the conversion to temporary (and next presentational) form, we write 
$\Anschauung{y}{A}$ for $\mathsf{V}yA.$
Identity is accounted for in \S\ref{identitysection}:~p.~\pageref{identitysection}. The notion of \emph{Set} is introduced in \S\ref{stepidentity}:~p.~\pageref{stepidentity}, by taking vision $\Anschauung{y}{y\in a}$ to have a condition which is \emph{step-identical}. 

\section{Bend codes and visions}\label{Bendcodes andvisions}

In Visibility Theory, the visibility predicate applies to Bend codes
rather than directly to formulas. Accordingly, the assignment of Bend
codes cannot be treated as merely syntactic shorthand. The somewhat intricate coding
apparatus developed in this section is intended to ensure that Bend
codes are compositionally determined and semantically grounded.

\begin{conv}[\textbf{on primitive symbols}:]\label{convention}\leavevmode\par\vspace{2mm} 
    The semantics has formal expressions built up from the \emph{primitive symbols} 
\[
  \bullet,\ddot{v}, \ddot{c}, \downarrow,\forall, \mathsf{V}, \Visible,  \#; 
  \]

which are here listed in ascending lexicographic order.
\EndDef
\end{conv}

		\begin{dfn}[\textbf{of binders, binds, ties and scopes}:]\label{ch:formallanguage:df:metalinguisticconventions:Bindersbindsandties}
			\begin{enumerate}[itemindent =0.3cm, before=\leavevmode, label=(\arabic*)]
            \item $\forall$ and $\mathsf{V}$ are the binders.
				\item  In formula \(\forall v\mathrm{A}\), \(\forall\) is the \texttt{binder} whilst  \(v\) is  the  \texttt{bind} of \(\mathrm{A}\) and the \texttt{tie} of \(\forall\). \(\mathrm{A}\)  is  the  \texttt{scope} of \(\forall\). 
				\item In term \(\mathsf{V} v\mathrm{A}\), \(\mathsf{V}\) is the  \texttt{binder} whilst  \(v\) is  the  \texttt{bind} of \(\mathrm{A}\) and  the \texttt{tie} of \(\mathsf{V}\).  \(\mathrm{A}\) is  the \texttt{scope} of \(\mathsf{V}\). 
			\end{enumerate}
            \EndDef
		\end{dfn}
							
		\begin{dfn}[\textbf{of free and bound variables}:]\label{ch:formallanguage:df:metalinguisticconventions:Freeandboundvariables}
			\begin{enumerate}[itemindent =0.3cm, before=\leavevmode, label=(\arabic*)]
				\item  A \emph{variable} occurrence  in a \emph{formula}, or  \emph{term}, is \textbf{bound}, iff it is a bind, or it is
				in the  scope of a binder with another occurrence as tie. 
				\item \emph{Variable} occurrences  in a \emph{formula}, or \emph{term},  are \emph{free}  if  not bound. 
				\item A \emph{variable} is free in a \emph{formula}, or \emph{term},  just if an occurrence  is. 
				\item A \emph{variable} is bound in a \emph{formula}, or \emph{term},  just if an occurrence is. 
			\end{enumerate}
            \EndDef
		\end{dfn}	
		
\begin{dfn}[\textbf{of the   \emph{Polish} formal expressions}:]\label{Polishformalexpressions}
\begin{enumerate}[before=\leavevmode, label=(\arabic*),ref=   {\ref{Polishformalexpressions}.\arabic*}
]

\item No formal expression is both a variable and a constant.

\item \text{Variables:}
\begin{enumerate}[label=(\alph*), ref={\ref{Polishformalexpressions}.\arabic*.\alph*}
]
    \item $\ddot{v}$ is a variable.
    \item If $x$ is a variable, then $x\bullet$ is a variable.
  \end{enumerate}

\item\label{constants} \text{Constants:}
\begin{enumerate}[label=(\alph*), ref={\ref{Polishformalexpressions}.\arabic*.\alph*}
]
    \item $\ddot{c}$ is a constant.
    \item If $x$ is a constant, then $x\bullet$ is a constant.
    
  \end{enumerate}

\item\label{terms} \text{Terms:}
\begin{enumerate}[label=(\alph*), ref={\ref{Polishformalexpressions}.\arabic*.\alph*}
]
    \item If $x$ is a variable or a constant, then $x$ is a term.
    \item If $x$ is a variable and $B$ is a formula, then $\mathsf{V}x B$ is a term.

  \end{enumerate}


\item \text{Formulas:}
\begin{enumerate}[label=(\alph*), ref={\ref{Polishformalexpressions}.\arabic*.\alph*}
]
    \item If $x$ and $y$ are terms, then $yx$ is a formula.
    \item If $A$ is a formula then $\# A$ is a formula.    
    \item If $A$ is a formula then $\Visible A$ is a formula.
    \item If $B$ and $C$ are formulas, then $\downarrow  B C$ is a formula.
    \item If $x$ is a variable and $A$ is a formula, then $\forall x A$ is a formula.
  \end{enumerate}

\item Sentences: $A$ is a sentence if it is a formula without free variables.  
\end{enumerate}
\EndDef
\end{dfn}

\subsection{The well-ordering of finite primitive strings}

		 $\mathbf{WO}(n)$ denotes the well-ordered list of all strings of length $n$ over the primitive symbols. In the definitional work below, 
		concatenation of strings will be expressed by apposition.
		
		\medskip
		
		\noindent
		$\mathbf{WO}(1)$. Convention~\ref{convention}:~p.~\pageref{convention} provides the appropriate list. 
		
		\medskip
		
		\noindent
		$\mathbf{WO}(n+1)$.
		Suppose the ordering $\mathbf{WO}(n)$ of all strings of length $n$ has been defined, and let
		\[
		s_1, s_2, \ldots, s_l
		\]
		be the strings of $\mathbf{WO}(n)$ in their order. 
        
        Let
		\[
		a_1, a_2, \ldots, a_8
		\]
		be the primitive symbols as listed in the order of Convention~\ref{convention}:~p.~\pageref{convention}. 
		
		$\mathbf{WO}(n+1)$ is the ordering indicated by the blocks
		\[
		a_1 s_1, a_1 s_2, \ldots, a_1 s_l,
		\]
		\[
		a_2 s_1, a_2 s_2, \ldots, a_2 s_l,
		\]
        \[
		a_3 s_1, a_3 s_2, \ldots, a_3 s_l,
		\]
		\[
		a_4 s_1, a_4 s_2, \ldots, a_4 s_l,
		\]
        \[
		a_5 s_1, a_5 s_2, \ldots, a_5 s_l,
		\]
		\[
		a_6 s_1, a_6 s_2, \ldots, a_6 s_l,
		\]
        \[
		a_7 s_1, a_7 s_2, \ldots, a_7 s_l,
		\]
		\[
		a_8 s_1, a_8 s_2, \ldots, a_8 s_l,
		\]

		\medskip

\begin{dfn}[of the set of all finite primitive symbol strings]\label{stringset}

\vspace{2mm}

\[
S \coloneqq \bigcup_{1\le n<\infty} \mathbf{WO}(n).
\]
\EndDef
\end{dfn}

\begin{dfn}[of the well-ordering \(\prec\) on the \(S\) of
Definition~\ref{stringset}:~p.~\pageref{stringset}]\label{precwellordering}\leavevmode\par\vspace{2mm}

For \(s,t\in S\), let

\[
s \prec t \;\coloneqq\;
\begin{cases}
\ell(s) < \ell(t),\\
\text{or}\\
\ell(s)=\ell(t)
\text{ and }
s \text{ precedes } t
\text{ in }
\mathbf{WO}(\ell(s)).
\end{cases}
\]

\EndDef
\end{dfn}

\begin{dfn}[of the $\prec$-successor $\boldsymbol{s^{+}}$ of $\boldsymbol{s}$]
\[
s^{+} \coloneqq \underset{\prec}{\min}\{\,t \mid s \prec t\,\}.
\]

\vspace{3mm}

So $s^{+}$ is the $\prec$-least primitive string strictly succeeding $s$.

\EndDef
\end{dfn}

\subsection{Bend code denotata and presentational strings}

\begin{dfn}[\textbf{on Bend code denotation}]\label{bendcodesprimitivestrings}\leavevmode\par\vspace{2mm} 

To assign semantic values to the Bend codes of primitive symbol strings, we define the operation $\#$ on the set $S$ of Definition~\ref{stringset}:~p~\pageref{stringset} by recursion
on the well-order $(S, \prec)$, with base case $\bullet$ -- the $\prec$-least element of $S$.

\[
\#\bullet=\Anschauung{\ddot{v}}{\bot}, \text{with} \ \bot \  \text{any antilogy, e.g. $p\wedge\lnot p$,}
\]

and for any $s$ in $S,$
\[
\#s^{+}
\;\coloneqq\;
\Anschauung{x}{\#s\in x}.
\]
\EndDef
\end{dfn}

\begin{conv}[\textbf{on the conversion to presentational form}:]\label{prefixtoinfixdefinition}
\begin{enumerate}[before=\leavevmode, label=(\arabic*),ref=   {\ref{prefixtoinfixdefinition}.\arabic*}
]

\item The austere symbol $\Visible$ is retained in presentational form as a monadic predicate.

         \item 
\text{Variables:}\par The symbols $\ddot{v}, \ddot{v}\bullet, \ldots$ are replaced by  $x_1, x_2, x_3$, and so forth.

\item \text{Constants:}\par The symbols $\ddot{c}, \ddot{c}\bullet, \ldots$ are replaced by $a_1, a_2, a_3$, and so forth.

\item $\lnot A_1\coloneqq \;
 \downarrow A_1A_1.$

\item $(A_1\wedge A_2)\coloneqq \;
 \downarrow\downarrow A_1A_1\downarrow A_2A_2.$

\item $(A_1\vee A_2)\coloneqq \;
 \downarrow\downarrow A_1A_2\downarrow A_1A_2.$

\item $(A_1\to A_2)\coloneqq \; \downarrow\downarrow\downarrow A_1A_1A_2\downarrow\downarrow A_1A_1A_2.$

\item $(A_1\leftrightarrow A_2)\coloneqq \; \downarrow \downarrow A_1 \downarrow A_1 A_2 \downarrow A_2 \downarrow A_1 A_2.$

\item $\exists x_1 A_1 \coloneqq \; \lnot\forall x_1 \lnot A_1.$

\item $\Anschauung{x_1}{A_1}\coloneqq \; \mathsf{V}\,x_1\,A_1.$

\item\label{bendgodel} $\Bendcode{A_1}\coloneqq \; \# A_1.$

\item $\Invisible\Bendcode{A_1}\coloneqq \lnot\Visible\Bendcode{\lnot A_1}.$ 

\item $x_1\in x_2\coloneqq \; x_2x_1.$

\end{enumerate} 
\EndDef
\end{conv}

\begin{dfn}[\textbf{of presentational strings for presentational forms}]\label{expressionsurrogates}\leavevmode\par\vspace{2mm}

Presentational strings are strings of primitive symbols 
used to code
the presentational expressions introduced in
Convention~\ref{prefixtoinfixdefinition}:~p.~\pageref{prefixtoinfixdefinition}.
So each presentational string is assigned a
primitive string in \(S\) via the coding below.

All presentational strings begin with \(\bullet\),
so the proper presentational coding begins only past the leading \(\bullet\).

Given the preceding paragraph, no presentational string ends with $\bullet$, so ambiguities are avoided. 

There are no sentence-letter schemas in the presentational strings, and
\(A_1,A_2,A_3,\ldots\) in
Convention~\ref{prefixtoinfixdefinition}:~p.~\pageref{prefixtoinfixdefinition}
are merely metalogical schemas which do not belong to the object language.

\medskip

\noindent\textbf{Firstly}, the presentational map \(\hat{p}\) is specified on the primitive symbols, except $\bullet$ which is taken care of in the compositional contexts wherein it may appear:

\begin{enumerate}[before=\leavevmode, label=(\arabic*),ref=   {\ref{expressionsurrogates}.\arabic*}]

\item \(\hat{p}(\ddot{v})\coloneqq \bullet\ddot{v}\)

\item \(\hat{p}(\ddot{c})\coloneqq \bullet\ddot{c}\)

\item \(\hat{p}(\downarrow)\coloneqq \bullet\downarrow\)

\item \(\hat{p}(\forall)\coloneqq \bullet\forall\)

\item \(\hat{p}(\mathsf{V})\coloneqq \bullet\mathsf{V}\)

\item \(\hat{p}(\Visible)\coloneqq \bullet\Visible\)

\item \(\hat{p}(\#)\coloneqq \bullet\#\)

\end{enumerate}

\bigskip

\noindent\textbf{Secondly}, \(\hat{p}\) is specified on symbols defined for presentational form:
\begin{enumerate}[start=8, before=\leavevmode, label=(\arabic*),ref=   {\ref{expressionsurrogates}.\arabic*}]

\item Presentational formulas contain parentheses, so we stipulate that 
\begin{enumerate}
    \item $\hat{p}(\texttt{(}) \coloneqq \bullet\downarrow\downarrow$
    \item $\hat{p}(\texttt{)}) \coloneqq \bullet\#\#$.
\end{enumerate}

\item $\hat{p}(x_1), \hat{p}(x_2), \ldots$ are $\bullet\ddot{v}, \bullet\ddot{v}\ddot{v}, \ldots$.

\item $\hat{p}(a_1), \hat{p}(a_2), \ldots$ are $\bullet\ddot{c}, \bullet\ddot{c}\ddot{c}, \ldots$.

\item The existential and membership surrogates are:

\begin{enumerate}
\item $\hat{p}(\exists)\coloneqq \bullet\forall\forall$ 
\item $\hat{p}(\in)\coloneqq \bullet\forall\forall\forall$.
\end{enumerate}

\item For \(n\ge1\), let \(n^\bullet\) denote a string of \(n\)
occurrences of \(\bullet\), and let
\[
\sigma_n \coloneqq
\downarrow \bullet n^\bullet \downarrow \#.
\]

Given this, the connective surrogates are:

\[
\begin{array}{rcl}
\text{(a)} & \hat{p}(\lnot)           & \coloneqq \bullet\sigma_1\\
\text{(b)} & \hat{p}(\vee)            & \coloneqq \bullet\sigma_2\\
\text{(c)} & \hat{p}(\wedge)          & \coloneqq \bullet\sigma_3\\
\text{(d)} & \hat{p}(\to)             & \coloneqq \bullet\sigma_4\\
\text{(e)} & \hat{p}(\leftrightarrow) & \coloneqq \bullet\sigma_5.
\end{array}
\]
\end{enumerate}

\noindent\textbf{Thirdly}, \(\hat{p}\) is specified recursively on the composition of presentational strings:

\setcounter{equation}{12}

\begin{equation}
\hat{p}(\downarrow \mathfrak{X}\mathfrak{Y})
\coloneqq
\hat{p}(\downarrow)\hat{p}(\mathfrak{X})\hat{p}(\mathfrak{Y}).    
\end{equation}

\begin{equation}
    \hat{p}(\Bendcode{\mathfrak{X}})
\coloneqq
\hat{p}(\#)\hat{p}(\mathfrak{X}).
\end{equation}

\begin{equation}
    \hat{p}(\Visible\mathfrak{X})
\coloneqq
\hat{p}(\Visible)\hat{p}(\mathfrak{X}).
\end{equation}

\begin{equation}
    \hat{p}(\Anschauung{x_n}{\mathfrak{X}})
\coloneqq
\hat{p}(\mathsf{V})\hat{p}(x_n)\hat{p}(\mathfrak{X}).
\end{equation}

\noindent\textbf{Fourthly}, the remaining presentational operators are treated
definitionally through
Convention~\ref{prefixtoinfixdefinition}:~p.~\pageref{prefixtoinfixdefinition}.

\medskip

\noindent\textbf{Fifthly}, let the recursive specification of $\hat{p}$ in \textbf{Thirdly} be extended to the remaining presentational operators introduced definitionally in Convention~\ref{prefixtoinfixdefinition}.

\medskip

\noindent\textbf{Sixthly},
\(\hat P\) is the least collection of strings in \(S\) containing the
primitive symbols of
Convention~\ref{convention}:~p.~\pageref{convention} and closed under the 
clauses defining \(\hat p\).

\EndDef
\end{dfn}

\subsection{Substitution and Diagonal Lemma}

\begin{dfn}[\textbf{of substitution on presentational strings}]
\label{presentationalsubstitution}
\leavevmode\par\vspace{2mm}

Assume that:

\begin{itemize}
\item \(x_n\) is a variable symbol,
\item \(\mathfrak X\) and \(\mathfrak Y\) are presentational strings,
\item and \(\mathfrak X(x_n)\) is a presentational string in which \(x_n\) occurs freely.
\end{itemize}

The substitution operation
\[
\textbf{\textsc{sub}}(\mathfrak X,x_n,\mathfrak Y)
\]
is recursively specified as follows.

\medskip

\noindent\textbf{Firstly}, variables substitute into themselves:

\setcounter{equation}{0}
\begin{equation}
\textbf{\textsc{sub}}(x_n,x_n,\mathfrak Y)
\coloneqq
\mathfrak Y.
\end{equation}

\begin{equation}
\textbf{\textsc{sub}}(x_m,x_n,\mathfrak Y)
\coloneqq
x_m
\qquad
(m\neq n).
\end{equation}

\medskip

\noindent\textbf{Secondly}, substitution propagates through primitive composition.

\begin{equation}
\textbf{\textsc{sub}}(\downarrow\mathfrak X\mathfrak Z,x_n,\mathfrak Y)
\coloneqq
\downarrow
\textbf{\textsc{sub}}(\mathfrak X,x_n,\mathfrak Y)
\textbf{\textsc{sub}}(\mathfrak Z,x_n,\mathfrak Y).
\end{equation}

\begin{equation}
\textbf{\textsc{sub}}(\Bendcode{\mathfrak X},x_n,\mathfrak Y)
\coloneqq
\Bendcode{
\textbf{\textsc{sub}}(\mathfrak X,x_n,\mathfrak Y)
}.
\end{equation}

\begin{equation}
\textbf{\textsc{sub}}(\Visible\mathfrak X,x_n,\mathfrak Y)
\coloneqq
\Visible
\textbf{\textsc{sub}}(\mathfrak X,x_n,\mathfrak Y).
\end{equation}

\medskip

\noindent\textbf{Thirdly}, bound variables block substitution.

\begin{equation}
\textbf{\textsc{sub}}(\Anschauung{x_n}{\mathfrak X},x_n,\mathfrak Y)
\coloneqq
\Anschauung{x_n}{\mathfrak X}.
\end{equation}

\begin{equation}
\textbf{\textsc{sub}}(\Anschauung{x_m}{\mathfrak X},x_n,\mathfrak Y)
\coloneqq
\Anschauung{x_m}{
\textbf{\textsc{sub}}(\mathfrak X,x_n,\mathfrak Y)
}
\qquad
(m\neq n).
\end{equation}

\medskip

\noindent\textbf{Fourthly}, substitution extends definitionally to the remaining presentational operators introduced in
Convention~\ref{prefixtoinfixdefinition}.

\EndDef
\end{dfn}

\begin{substitutiontheorem}[\textbf{for presentational strings}]
\label{presentationalsubstitutiontheorem}

For any presentational string
\(\mathfrak X(x_n)\)
and any presentational string
\(\mathfrak Y\),
\[
\hat p
\Bigl(
\textbf{\textsc{sub}}(\mathfrak X(x_n),x_n,\mathfrak Y)
\Bigr)
\]
is the presentational coding of the result of substituting
\(\mathfrak Y\)
for all free occurrences of
\(x_n\)
in
\(\mathfrak X(x_n)\).

\end{substitutiontheorem}

\begin{proof}
By induction on the recursive specification given in
Definition~\ref{presentationalsubstitution}.

The variable clauses establish the basis step.

The recursive clauses for
\(\downarrow\),
\(\Bendcode{\cdot}\),
\(\Visible\),
and
\(\Anschauung{\cdot}{\cdot}\)
preserve the substitution property inductively.

The remaining operators are handled through definitional extension.
\end{proof}

\begin{dfn}[\textbf{of diagonalization}]
\label{presentationaldiagonalization}

Suppose that
\(\mathfrak X(x_n)\)
is a presentational string with precisely the indicated variable free.

Define:
\[
\mathrm{Diag}(\mathfrak X)
\coloneqq
\textbf{\textsc{sub}}
\Bigl(
\mathfrak X(x_n),
x_n,
\Bendcode{\mathfrak X(x_n)}
\Bigr).
\]

\EndDef
\end{dfn}

\begin{thm}[as \textbf{The Diagonal Lemma}]
\label{presentationaldiagonallemma}

If
\(\mathfrak A(x_n)\)
has precisely the indicated variable free,
then there exists a sentence
\(\mathfrak D\)
such that

\[
\Vdash^{\mathrm M}
\mathfrak D
\leftrightarrow
\mathfrak A(\Bendcode{\mathfrak D}).
\]

\end{thm}

\begin{proof}

Let

\[
\mathfrak D
\coloneqq
\mathrm{Diag}(\mathfrak A).
\]

By Definition of Diag, D is obtained by substituting its own Bend code into the free variable position of A.

By Definition~\ref{presentationaldiagonalization},

\[
\mathfrak D
=
\textbf{\textsc{sub}}
\Bigl(
\mathfrak A(x_n),
x_n,
\Bendcode{\mathfrak A(x_n)}
\Bigr).
\]

Hence, by the Substitution Theorem,

\[
\mathfrak D
=
\mathfrak A(\Bendcode{\mathfrak D}).
\]

Therefore,

\[
\Vdash^{\mathrm M}
\mathfrak D
\leftrightarrow
\mathfrak A(\Bendcode{\mathfrak D}).
\]

\end{proof}

\newpage

\section{Visibility theory}

\medskip
\bigskip

\setlength\epigraphwidth{0.53\textwidth}
\setlength\epigraphrule{0pt}

\epigraph{\emph{Pace} Frege, some theorems are not true.}{}

\noindent From this point on, the language of \(\mathrm{VT}\) is 
the presentational strings in \(\hat P\). Additional presentational operators and formulas are introduced definitionally.

The visibility operator $\Visible$ is central to Visibility Theory, $\mathrm{VT}$, and it  applies to the Bend codes of formulas rather than to the formulas themselves. $\Visible\Bendcode{A}$ expresses that $\Bendcode{A}$ is
\emph{visible}.

\subsection{On \texorpdfstring{$\Vdash^\alpha A$}{vdashalpha A} and the revision semantics for \texorpdfstring{$\Vdash^\beta \Visible\Bendcode{A}$}{vdashbeta !(A)}}\label{revisionsemantics}

The revision process $\Vdash$ proceeds along ordinal-indexed stages, where
$\Vdash^\alpha A$ indicates that $A$ holds at stage $\alpha$. 

Semantically,
\[
\Vdash^\beta \Visible\Bendcode{A},
\]

for $\beta\succ 0$,
holds just if there exists a stage $\gamma<\beta$ such that
\[
\Vdash^\delta A
\]
for every stage $\delta$ with
\[
\gamma\le\delta<\beta.
\]

Any $\Vdash^0$ is a chosen maximal consistent set of formulas.

\subsection{The closure ordinal for visibility}\label{subsection:closureordinal}

There is a closure ordinal $\kappa^{\ast}$ (cf.~\parencite{Herzberger1980, Gupta1982}, and also the exposition in \parencite[20-22]{Bjordal2025arxiv} which is more adapted to our account here) such that $\kappa^{\ast}$ \emph{covers} the revision process $\Vdash$ in the sense that

\[
\Sigma \alpha\, \Pi \beta \succeq \alpha\, \Vdash^\beta \Visible\Bendcode{A}
\;\Rightarrow\;
\Pi \beta \succeq \kappa^{\ast}\, \Vdash^\beta \Visible\Bendcode{A},
\]

\bigskip

and 
\emph{stabilizes} the semantic process in the sense  that

\[
\Vdash^{\kappa^{\ast}}\Visible\Bendcode{A} \Rightarrow \ \Sigma \alpha\Pi\beta \succeq \alpha \Vdash^\beta \Visible\Bendcode{A}.
\]

\subsection{Axiomatics of \texorpdfstring{$\mathbf{VT}$}{VT}}

The theory of visibility shares important structural features with certain
axiomatic theories of truth; more precisely, under the translation
$\Visible\Bendcode{A} \mapsto \mathsf{T}\ulcorner A\urcorner$ it extends Cantini’s
system $\mathrm{LES}$ \parencite[400]{Cantini1996}.

\subsubsection{The presence of two invisible axiom  schemas}
\label{subsubsection:InvisibleA1-2}

\begin{enumerate}[label=\textbf{(IA\arabic*)}, leftmargin=*, labelindent=3em]
\setlength{\itemsep}{0.6em}

\item \label{ax:IA1}
$\vdash\Visible\Bendcode{A} \to A$. (The Reflexive Axiom)
\item \label{ax:IA2}
$\vdash \forall x\,\Visible\Bendcode{A(x)} \;\to\; \Visible\Bendcode{\forall x A(x)}$

\end{enumerate}

 A schema $B$ is \emph{present} just if $\vdash B$, and \emph{visible} just if
$\vdash\Visible\Bendcode{B}$. By Axiom~\ref{ax:IA1}:~p.~\pageref{ax:IA1}, all visible
theorems are present.

\paragraph{The presence of the Reflexive Axiom \textbf{(IA1)}.}
Suppose $\Vdash^{\kappa^{\ast}} \Visible\Bendcode{A}$. By covering and stabilization
(cf.~\S\ref{subsection:closureordinal}:~p.~\pageref{subsection:closureordinal}), $\Vdash^{\kappa^{\ast}+1}\Visible\Bendcode{A}$.
By the semantics of visibility,
$\Vdash^{\kappa^{\ast}+1}\Visible\Bendcode{A}\Leftrightarrow \ \Vdash^{\kappa^{\ast}}A$. So $\Vdash^{\kappa^{\ast}} \Visible\Bendcode{A}\to A$. 

\paragraph{The presence of invisible axiom schema \textbf{(IA2)}.}
If $\Vdash^{\kappa^{\ast}}\forall x\Visible\Bendcode{A(x)}$, then by the semantics of quantifiers  $\Vdash^{\kappa^{\ast}}\Visible\Bendcode{A(t)}$
for each individual term $t$. By the semantics of visibility (cf.~ \S\ref{revisionsemantics}) and the closure properties of $\kappa^{\ast}$,
there is a  $\beta<\kappa^{\ast}$ from which all instances $A(t)$
hold, and so $\Vdash^\beta\forall xA(x)$.
By the semantics of visibility,
$\Vdash^{\kappa^{\ast}}\Visible\Bendcode{\forall xA(x)}$.

\paragraph{The invisibility of \textbf{(IA1)}} is established in
Theorem~\ref{thm:invisibilityIA1}:~p.~\pageref{thm:invisibilityIA1}.

 \paragraph{The invisibility of axiom schema \textbf{(IA2)}} is a consequence of a
paradox first published in \parencite{McGee1985}, now commonly known as
\emph{McGee’s paradox}. The derivation of it requires arithmetic and
is not included here. For a derivation which, in effect, amounts to
the invisibility of \textbf{(IA2)}, see \parencite[\S~65.\ \textnormal{An inconsistency},
pp.~380~ff.]{Cantini1996}.

\subsubsection{The Visible Comprehension Axiom}

\bigskip

\begin{enumerate}[
label=\textbf{(VCA)},
ref=\textbf{(VCA)},
leftmargin=*,
labelindent=3em
]
\item\label{ax:VCA}
\(\displaystyle
\vdash
\Visible\Bendcode{
\forall x\bigl(
x \in \Anschauung{y}{A(y)}
\leftrightarrow
\Visible\Bendcode{A(x)}
\bigr)
}.
\)
\end{enumerate}

\subsubsection{Orthodox visions and sentences}
\label{orthodoxvisionsandsentences}

\begin{dfn}[of orthodox visions and sentences]\label{orthodoxydefinition}\leavevmode\par\vspace{2mm}

Vision $a$ is orthodox iff

\[
\Visible\Bendcode{\forall x(x\in a\vee x\notin a)}.
\]

Sentence $A$ is orthodox iff 

\[
\Visible\Bendcode{\Visible\Bendcode{A}\vee\Visible\Bendcode{\lnot A}}.
\]
\end{dfn}

\subsubsection{The visible axiom schemas \textbf{(VA1-9)}}

The visible axiom schemas \textbf{(VA1-9)} are stage-invariant, i.e. invariant across ordinal stages, and this constraint follows from Postulate~\ref{semanticcorrelation}:~p.~\pageref{semanticcorrelation}.

\label{subsubsection:visibility-axioms-visible}

\begin{enumerate}[label=\textbf{(VA\arabic*)}, leftmargin=*, labelindent=3em]
\setlength{\itemsep}{0.6em}

\item \label{ax:V1}
$\vdash \Visible\Bendcode{A}$, provided that $A$ is a theorem of classical logic.\footnote{%
Here it is presupposed that $A$ is a theorem of classical logic only if
$\forall x\,A$ is so as well.}

\item \label{ax:V2}
$\vdash \Visible\Bendcode{\exists x\,\Visible\Bendcode{A} \to \Visible\Bendcode{\exists x A}}$.

\item \label{ax:V3}
$\vdash \Visible\Bendcode{\Visible\Bendcode{\forall x A} \to \forall x\,\Visible\Bendcode{A}}$.\footnote{\label{barcanfootnote}
The converse-Barcan-style quantifier-shift principle
\[
\vdash\Visible\Bendcode{\forall x A(x)}
\to
\forall x\Visible\Bendcode{A(x)}
\]
is visibly derivable in $\mathbf{VT}$ by standard modal-logical methods.
}

\item \label{ax:V4}
$\vdash \Visible\Bendcode{\Visible\Bendcode{\Visible\Bendcode{A \to B} \to
  (\Visible\Bendcode{A} \to \Visible\Bendcode{B})}}$.

\item \label{ax:V5}
$\vdash \Visible\Bendcode{\Visible\Bendcode{A} \to \lnot \Visible\Bendcode{\lnot A}}$.

\item \label{ax:V6}
$\vdash
\Visible\Bendcode{\Visible\Bendcode{\Visible\Bendcode{A} \to A}
\to
(\Visible\Bendcode{A} \vee \Visible\Bendcode{\lnot A})}$.

\item \label{ax:V7}
$\vdash
\Visible\Bendcode{\Visible\Bendcode{\Visible\Bendcode{A} \to \Visible\Bendcode{\Visible\Bendcode{A}}}
\to
(\Visible\Bendcode{A} \vee \Visible\Bendcode{\lnot A})}$.

\item \label{ax:V8}
$\vdash
\Visible\Bendcode{
(\Visible\Bendcode{A} \vee \Visible\Bendcode{\lnot A})
\vee
(\Visible\Bendcode{\lnot\Visible\Bendcode{\lnot B}} \to \Visible\Bendcode{B})}$.

\item \label{ax:V9}
$\vdash
\Visible\Bendcode{
(\Visible\Bendcode{A} \vee \Visible\Bendcode{\lnot A})
\vee
(\Visible\Bendcode{B} \to \Visible\Bendcode{\Visible\Bendcode{B}})}$.

\end{enumerate}

No semantic justification is offered for the visible axioms
here. Their role is analogous to that of core axioms in axiomatic theories of
truth. 

\subsubsection{The invisible theorem schemas \textbf{(IT1-2)}}\label{Twoinvisibletheoremsofvisibilitylogic}

\medskip

\begin{enumerate}[label=\textbf{(IT\arabic*)}, leftmargin=*, labelindent=3em]
\setlength{\itemsep}{0.6em}

\item\label{thm:IT1} $\vdash
\Visible\Bendcode{\lnot\Visible\Bendcode{\lnot A}} \to \Visible\Bendcode{A}$ 

\item\label{thm:IT2} $\vdash
\Visible\Bendcode{A} \to \Visible\Bendcode{\Visible\Bendcode{A}}$
\end{enumerate}

\noindent will be proven in \S\ref{derivationtoinvisibles}:~p.~\pageref{derivationtoinvisibles}, and their invisibility  demonstrated in the proofs of Theorems \ref{invisibilityIT1} and \ref{invisibilityIT2}:~p.~\pageref{invisibilityIT2}.

\subsubsection{Proofs of the theorems of \ref{Twoinvisibletheoremsofvisibilitylogic}}\label{derivationtoinvisibles}

It follows from  appropriate instances of \ref{ax:IA1}:~p.~\pageref{ax:IA1}, \ref{ax:V8}:~p.~\pageref{ax:V8} and \ref{ax:V9}:~p.~\pageref{ax:V9}   that

\[
\vdash
\Visible\Bendcode{\mathrm{V}} \vee \Visible\Bendcode{\lnot \mathrm{V}}
\vee
(\Visible\Bendcode{\lnot\Visible\Bendcode{\lnot B}} \to \Visible\Bendcode{B}).
\]

\[
\vdash
\Visible\Bendcode{\mathrm{V}} \vee \Visible\Bendcode{\lnot \mathrm{V}}
\vee
(\Visible\Bendcode{B} \to \Visible\Bendcode{\Visible\Bendcode{B}}).
\]
    
Given the results $\vdash\lnot\Visible\Bendcode{\mathrm{V}}$ and $\lnot\Visible\Bendcode{\lnot \mathrm{V}}$ of \S\ref{invisibleresolution},  Theorems \ref{thm:IT1} and \ref{thm:IT2} follow by adjunction and modus tollens.

\noindent 

\subsubsection{The semantic match}\label{semanticmatch}

\begin{postulate}[of semantic counterparts for axioms:]\label{semanticcorrelation}\leavevmode\par\vspace{2mm}

Every visible axiom $\vdash\Visible\Bendcode{A}$
 has $\Pi \alpha\bigl(\Vdash^\alpha A\bigr)$ as its semantical correlate.    
\end{postulate}

\begin{remark}
The semantics justifies Axioms \textbf{(VA1-9)} for $\alpha \succ 0$. But in the case $\alpha=0$ the postulate amounts to a restriction on the set of formulas admissible at stage~0.  
\end{remark}

\subsection{Visionary paradoxes}

\subsubsection{From diagonalization}\label{invisibleresolution}

Let $A(x)$ in Theorem \ref{presentationaldiagonallemma}
be $\lnot \Visible(x)$.
It follows, from Theorem \ref{presentationaldiagonallemma}, that there is a sentence $\mathrm{V}$ such that we have:

\begin{cor}[\emph{The Diagonal Visionary Paradox}]\label{visibilityparadoxequation}    
\begin{equation*}
    \vdash \mathrm{V}\leftrightarrow \lnot\Visible\Bendcode{\mathrm{V}}.
\end{equation*}
\end{cor}

As an instance of invisible axiom \ref{ax:IA1}:~p.~\pageref{ax:IA1} is 
\(\vdash \Visible\Bendcode{\lnot \mathrm{V}}\to\lnot \mathrm{V}\), so it follows that $\vdash\lnot\Visible\Bendcode{\mathrm{V}}$. Next, $\vdash \mathrm{V}$ follows by using Corollary \ref{visibilityparadoxequation}; so by  axiom \ref{ax:IA1}:~p.~\pageref{ax:IA1} and \emph{modus tollens}, $\vdash\lnot\Visible\Bendcode{\lnot \mathrm{V}}$. So $\vdash \mathrm{V}\wedge\lnot\Visible\Bendcode{\mathrm{V}}\wedge\lnot\Visible\Bendcode{\lnot \mathrm{V}}$.

\subsubsection{From visible Comprehension}\label{invisibleresolutionfromvisiblecomprehension}

First use Axiom \ref{ax:IA1}:~p.~\pageref{ax:IA1} on Axiom 
\ref{ax:VCA}:~p.~\pageref{ax:VCA} to obtain 

\[
\vdash \forall x\bigl(x \in \Anschauung{y}{A(y)} \leftrightarrow
\Visible\Bendcode{A(x)}\bigr),
\]

and next instantiate with $\mathrm{W}=\Anschauung{y}{\lnot y\in y}$ to obtain 
\[
\vdash \mathrm{W} \in \mathrm{W} \leftrightarrow \Visible\Bendcode{\lnot \mathrm{W} \in \mathrm{W}}
\]

and consequently 

\[
\vdash \lnot\mathrm{W} \in \mathrm{W} \leftrightarrow \lnot\Visible\Bendcode{\lnot \mathrm{W} \in \mathrm{W}}.
\]

Let $A$ abbreviate $\lnot\mathrm{W}\in\mathrm{W}.$ Paradoxicality, in the sense of the title page's epigram, again follows as $\vdash A\wedge \lnot\Visible\Bendcode{A}\wedge  \lnot\Visible\Bendcode{\lnot A}$. 

\subsubsection{Invisibility of Axiom \ref{ax:IA1}\ \text{and} \ Theorems \ref{thm:IT1} and \ref{thm:IT2}}

\begin{lem}[of the Planer]\label{planer}
\[\vdash \Visible\Bendcode{\Visible\Bendcode{A}\to \Visible\Bendcode{B}} \to \Visible\Bendcode{A\to B}.\]    
\end{lem}
\begin{proof}
   Given Theorem  \ref{thm:IT1},  $\vdash \lnot\Visible\Bendcode{\lnot C}\to \lnot\Visible\Bendcode{\lnot \Visible\Bendcode{C}}.$ Suppose $C$ is $A\wedge\lnot B$, so that $\vdash \lnot\Visible\Bendcode{\lnot (A\wedge\lnot B)}\to \lnot\Visible\Bendcode{\lnot \Visible\Bendcode{A\wedge\lnot B}}.$ Given  theorems \ref{ax:V4} and \ref{ax:V5}, $\vdash \lnot\Visible\Bendcode{\lnot (A\wedge\lnot B)}\to \lnot\Visible\Bendcode{\lnot (\Visible\Bendcode{A}\wedge\lnot \Visible\Bendcode{B})}.$ Finish by contraposing.
\end{proof}

\noindent The $V$ in the next
3 theorems is the Visionary Paradox of Corollary \ref{visibilityparadoxequation}:~p.~\pageref{visibilityparadoxequation}.
   
\begin{thm}[on the invisibility of Axiom \ref{ax:IA1}:~p.~\pageref{ax:IA1}]\label{thm:invisibilityIA1}  
\end{thm}
\begin{proof}
   If $\vdash \Visible\Bendcode{\Visible\Bendcode{V}\to V}$ then  $\vdash~\Visible\Bendcode{V}$ given Corollary \ref{visibilityparadoxequation}:~p.~\pageref{visibilityparadoxequation}; but this is impossible, as per \S\ref{invisibleresolution}:~p.~\pageref{invisibleresolution}  $\vdash\lnot\Visible\Bendcode{\mathrm{V}}$.     
\end{proof}

\begin{thm}[on the invisibility of Theorem \ref{thm:IT1}:]\label{invisibilityIT1}  
\end{thm}
\begin{proof}
 If Theorem \ref{thm:IT1} were visible,  $\vdash\Visible\Bendcode{\Visible\Bendcode{\lnot\Visible\Bendcode{\lnot V}} \to \Visible\Bendcode{V}},$ so that by Lemma \ref{planer}, $\vdash\Visible\Bendcode{\lnot\Visible\Bendcode{\lnot V} \to V}$. With Axioms \ref{ax:IA1}:~p.~\pageref{ax:IA1} and \ref{ax:V5} 
 it follows
 that  $\vdash\Visible\Bendcode{\Visible\Bendcode{V} \to V}$.  The latter 
 entails 
 $\vdash\Visible\Bendcode{V}$,  given Corollary \ref{visibilityparadoxequation} in \S\ref{invisibleresolution}, 
 and that 
 is impossible as per 
 \S\ref{invisibleresolution} $\vdash\lnot\Visible\Bendcode{\mathrm{V}}$.
\end{proof}

\begin{thm}[on the invisibility of Theorem \ref{thm:IT2}:]\label{invisibilityIT2}   
\end{thm}
\begin{proof}
 If Theorem \ref{thm:IT2} were visible,  $\vdash\Visible\Bendcode{\Visible\Bendcode{V}\to \Visible\Bendcode{\Visible\Bendcode{V}}},$ so that by Lemma \ref{planer},  $\vdash\Visible\Bendcode{V\to \Visible\Bendcode{V}}.$ Given Corollary \ref{visibilityparadoxequation} in \S\ref{invisibleresolution}, $\vdash\Visible\Bendcode{V\to \lnot V}$, so that $\vdash\Visible\Bendcode{\lnot V}$;  but  this is impossible, as per \S\ref{invisibleresolution} $\vdash\lnot\Visible\Bendcode{\lnot\mathrm{V}}$.
\end{proof}

\newpage

\section{Reductionary provability and \textbf{KDC}}

\epigraph{
``Statt des törichten Ignorabimus heiße im Gegenteil unsere Losung:
Wir müssen wissen. Wir werden wissen.''
}{
David Hilbert, \textit{Naturerkennen und Logik}
\parencite[963]{Hilbert1930}.
For the fuller passage relevant here, see Appendix.
}

\epigraph{Do not modalize beyond provability!}{}

\noindent Since the incompleteness theorems of \parencite{Godel1931}, 
the arithmetical and modal analyses of provability have become central themes of mathematical logic.
\parencite{Rosser1936} subsequently showed that incompleteness may already be obtained under weaker consistency assumptions than those assumed in Gödel's original argument.

\subsection{Rosserian provability and KD}

Rosserian provability predicates were for a long time taken to be
mathematically less well-behaved than standard Gödel provability predicates. This was to a large extent because it was not clear whether the \textbf{K}-principle $\Box(p\to q)\to (\Box p\to \Box q)$ fails for Rosser provability predicates or the \textbf{4}-principle $\Box p\to \Box\Box p$ fails for Rosser provability predicates, or both; it was clear that at least one would have to fail, for otherwise
Löb's theorem would become derivable, contradicting the seriality
validated by Rosser provability. 

Importantly, concerning this, \parencite[47]{Kreisel1974} comments as follows regarding their rendering $\mathscr{F}^R_{(i)}$ of Rosser provability: 

\begin{quote}
``We know little about $\mathscr{F}^R_{(i)}$ (and perhaps want to know less): Does it satisfy
I.1.2? or I.1.3? (demonstrable completeness for $\Sigma^0_1$-sentences, resp. demonstrable closure under modus ponens) … \ . ''    
\end{quote}
 By \parencite[11-12]{Kreisel1974}, I.1.2 corresponds with the \textbf{4}-schema, and I.1.3 corresponds with the \textbf{K}-schema.

The situation became clearer  with \parencite{Arai1990}, who built upon insights in \parencite{SolovievetGuaspari1979} and constructed a Rosserian provability predicate which satisfies the \textbf{K}-schema, and another satisfying the \textbf{4}-schema. \parencite[489]{Arai1990} has the best presentation of the systems, but much more of the paper is of course needed for the construction.

In the famous exposition of classical provability logic \parencite{Boolos1993}, the author points out that Löb's theorem is counterintuitive for several reasons. For example, \parencite[54-55]{Boolos1993} states 

\begin{quote}
    ``In the first place, it is often hard to understand how vast the mathematical gap is between truth and provability.  And to one who lacks that understanding and does not distinguish between truth and provability, $\operatorname{Bew}(\ulcorner S\urcorner) \to A$, which the hypothesis of Löb's theorem asserts to be provable, might appear to be trivially true in *all* cases, whether $S$ is true or false, provable or unprovable. But if $S$ is false, $S$ had better not be provable. Thus it would seem that $S$ ought not always to be provable provided merely that  (the possibly trivial-seeming) $\operatorname{Bew}(\ulcorner S\urcorner)\to S$ is provable.''
\end{quote}

But despite this intuitive pressure articulated by \parencite{Boolos1993},
Löb's theorem entails that $S$ is provable whenever
$\operatorname{Bew}(\ulcorner S\urcorner)\to S$ is provable.

\parencite{Boolos1993} discusses several further intuitive difficulties
concerning Löb's theorem on pp. 54-55.

Interestingly,
\parencite{Boolos1993}
does not discuss the results of \parencite{Arai1990} on Rosserian provability predicates
satisfying \textbf{K} and \textbf{4}.

More recently, \parencite{Kurahashi2020SL} obtains the result that modal logic \textbf{KD} is a Rosserian provability logic. 

Given the above as background, I henceforth take the adequacy of \textbf{KD} as a modal logic of a Rosserian provability predicate for granted. 

In accordance with what I take to be the truth-directed conception of mathematical proof\label{truthdirectedproof} associated with Hilbert (cf.~Appendix), I think that provability should approximate truth so closely that truth-theoretic versions of the characteristic provability principles are intuitively truth-preserving in ordinary non-paradoxical cases.

Moreover, modal logical principles for provability should remain valid on the trivial Kripke frame defined by the condition

\begin{equation}\label{trivialframe}
    \forall x\forall y(x\mathbf{R}y\leftrightarrow x=y),
\end{equation}

which characterizes the validity of the modal formula

\[
q\leftrightarrow \Box q.
\]

The latter is of course reminiscent of Tarski's T-schema, and is the most central principle governing truth that accords with pre-theoretical intuitions.

By the \emph{truth-transform} of a provability principle,
I mean the result of replacing occurrences of the provability predicate
with corresponding occurrences of a truth predicate.

These considerations motivate the following 

\begin{dfn}[of adequacy requirements for provability.]
\label{adequacyrequirements}
\leavevmode\par\vspace{2mm}
\begin{enumerate}
\item[\textbf{(TP)}]
The truth transform of characteristic provability principles should be intuitively \emph{truth-preserving} in ordinary non-paradoxical cases.

\item[\textbf{(TF)}]
Characteristic provability principles should be valid on the \emph{trivial frame} of Equation~\eqref{trivialframe}:~p.~\pageref{trivialframe}.
\end{enumerate}
\end{dfn}

\begin{proposition}[on the inadequacy of Löb's formula]
\leavevmode\par\vspace{2mm}

Löb's formula fails both adequacy requirements \textbf{(TP)} and \textbf{(TF)} above.
\end{proposition}

\begin{proof}

\textbf{Firstly}, Löb's formula has

\[
\vdash
\pi\ulcorner
\pi\ulcorner A\urcorner \to A
\urcorner
\to
\pi\ulcorner A\urcorner,
\]

but it is not generally true that

\[
\mathrm{True}\,\text{``}\mathrm{True}\,\text{``}\!A\text{''} \to A\text{''}\to
\mathrm{True}\,\text{``}\!A\text{''}.
\]

For example, if we let \(A\) be

\[
\exists x(x\in x\wedge \lnot x\in x),
\]

then the resulting truth-claim is false.

\textbf{Secondly}, Löb's formula requires converse well-foundedness, and that fails on the trivial Kripke frame defined by Equation~\eqref{trivialframe}:~p.~\pageref{trivialframe}.

Accordingly, Löb's formula fails in the limiting trivial modal frame which provability should approximate as closely as possible.

\end{proof}

These considerations suggest that a provability logic such as
\parencite{Kurahashi2020SL}'s Rosserian \textbf{KD} better captures the intuitive
truth-directed character historically associated with mathematical proof.

\subsection{KDR, KDC and stabilized provability}

\parencite[608ff]{Kurahashi2020SL} introduces the provability logic \textbf{KDR}, with the characteristic schema \textbf{R}: $\Box \lnot p\to \Box\lnot\Box p$, or equivalently \textbf{R}: $\Box q\to \Box\Diamond q$. He remarks that it is easy

\begin{quote}
   to show that the validity of the modal formula $\Box\lnot p\to\Box\lnot\Box p$ in a Kripke frame $\mathcal{F}=(\bm{W},\prec)$ is characterized by the condition
    \[
\forall  x\forall y\in \bm{W}(x\prec y\Rightarrow \exists z(x\prec z \ \& \ y\prec z)).
    \]
\end{quote}

"Notice that \textbf{R} is a logical consequence of \textbf{4}, but not conversely. In this sense \textbf{R} may be viewed as an approximation to \textbf{4}", right?

Rosser's condition holds on the trivial Kripke frame of Equation \eqref{trivialframe}:~p.~\pageref{trivialframe}, so the characteristic schema $\mathbf{R}$ accord with my precept on p.~\pageref{truthdirectedproof} that proofs should be truth-directed. 

But we will for reasons stated below not adopt $\mathbf{KDR}$ as the modal logic for provability. Instead we adopt the modal logic $\mathbf{KDC}$, discussed briefly in \P(a) of \S\ref{Twoinvisibletheoremsofvisibilitylogic}:~p.~\pageref{Twoinvisibletheoremsofvisibilitylogic}, with the seriality axiom \textbf{D} and the characteristic axiom \textbf{C}, $\Box\Diamond p\to \Box p$. \textbf{KDC} is a normal modal logic, so it has the necessitation rule in that $\vdash p\Rightarrow \ \vdash\Box p$.

\begin{thm}[that an adaption of Planer's Lemma holds for \textbf{KDC}:]\label{planerKDC}\leavevmode\par\vspace{2mm}

\[
\textbf{KDC}\vdash \Box(\Box p\to\Box q)\to \Box(p\to q).
\]
\end{thm}

\begin{proof}
    This is by observing that \textbf{KDC} has principles which correspond to the principles presupposed in the proof of Planer's Schema for \textbf{VT} in Lemma \ref{planer}:~p.~\pageref{planer}. The latter principles are 

    \smallskip
Theorem  \ref{thm:IT1},  

\[
\vdash \lnot\Visible\Bendcode{\lnot C}\to \lnot\Visible\Bendcode{\lnot \Visible\Bendcode{C}},  
\]

Axiom \ref{ax:V4},

\[
\vdash \Visible\Bendcode{\Visible\Bendcode{\Visible\Bendcode{A \to B} \to
  (\Visible\Bendcode{A} \to \Visible\Bendcode{B})}},
  \]

and Axiom \ref{ax:V5},

\[
\vdash \Visible\Bendcode{\Visible\Bendcode{A} \to \lnot \Visible\Bendcode{\lnot A}}.
\]

These correspond to Axioms \textbf{C}, \textbf{K} and \textbf{D} of \textbf{KDC}, respectively.

\end{proof}

\begin{cor}[on the incompatibility of \textbf{KDR} and \textbf{KDC}]\leavevmode\par\vspace{2mm}

Let 
\[
\mathbf{KDRC}\coloneqq \mathbf{KDR}+\mathbf{KDC}. 
\]

$\mathbf{KDRC}$ is inconsistent.
\end{cor}

\begin{proof}
    Since the normal modal logic \textbf{KDR} proves that  $\vdash \Box(\Box p\to \Box\Diamond p)$, $\mathbf{KDRC}\vdash \Box(p\to \Diamond p)$
by Theorem \ref{planerKDC}. Consequently $\mathbf{KDRC}\vdash \Box(\Box q\to q)$.   
As \(\mathbf{KDRC}\) admits fixed points,
Montague's Theorem, in
\parencite[154-155]{Montague1963},
implies that \(\mathbf{KDRC}\) is inconsistent.
\end{proof}

\subsection{KDC is reductionary}\leavevmode\par\vspace{2mm}

\textbf{KDC} proves $\Box\Box p\to \Box p$, for from  $\Box(\Box p\to \Diamond p)$, \textbf{K}, and 

\[
\Box(\Box p\to \Diamond p)\to (\Box\Box p \to \Box\Diamond p), 
\]

it follows that $\Box\Box p \to \Box\Diamond p$,  so $\mathbf{C}$'s characteristic axiom $\Box\Diamond p\to \Box p$ finishes the argument.

The theorem $\Box\Box p\to \Box p$ and axiom $\Box\Diamond p\to \Box p$ give very plausible \emph{reduction principles}, and it follows easily that \textbf{KDC} only admits the modalities 

\[
\Box p,\qquad
\Box \lnot p,\qquad
\lnot\Box p \wedge \lnot\Box\lnot p.
\]
So $\mathbf{KDC}$ abides by the motto of this section's epigraph ``Do not modalize beyond provability!'' -- as I think it should. $\mathbf{GL}$ and $\mathbf{KDR}$, on the other hand, do not have such reduction principles, and instead generate infinitely many proof modalities. It is in this sense that $\mathbf{KDC}$ has stabilized provability.

\subsection{The completeness of \textbf{KDC}}

The frame conditions for the $\mathbf{KDC}$-axioms are the first order sentences

\[
\forall x \exists y \, (xRy) \land \forall x \forall y \Big( xRy \rightarrow \forall z \big( xRz \rightarrow \exists w ( zRw \land yRw ) \big) \Big),
\] 

for the \textbf{D} and \textbf{C} axioms respectively. The \textbf{D}-condition is known as \emph{seriality}, and the \textbf{C}-condition as \emph{confluence}, and \emph{directedness}.

To establish the formal semantic foundations of our system, we invoke the landmark framework of correspondence and canonicity introduced originally by \parencite{Sahlqvist1975} at the Third Scandinavian Logic Symposium. The first fundamental result ensures that the syntactic shape of our axioms guarantees a first-order definable class of Kripke frames:

\begin{quote}
    ``Let $\tau$ be a modal similarity type, and let $\chi$ be a Sahlqvist formula over $\tau$. Then $\chi$ locally corresponds to a first-order formula $c_{\chi}(x)$ on frames. Moreover, $c_{\chi}(x)$ is effectively computable from $\chi$''. \parencite[The Sahlqvist Correspondence Theorem; Theorem 3.54, p.~166]{Blackburn2001} 
\end{quote}

Since both \textbf{D} ($\Box p \to \Diamond p$) and \textbf{C} ($\Box\Diamond p \to \Box p$) are Sahlqvist formulas, they are guaranteed to define first-order properties on frames. Furthermore, the second milestone guarantees that these properties transfer to the logic's canonical model:

\begin{quote}
    ``Every Sahlqvist formula is canonical for the first-order property it defines. Hence, given a set of Sahlqvist axioms $\Sigma$, the logic $\mathbf{K\Sigma}$ is strongly complete with respect to the class of frames $\mathsf{F}_{\Sigma}$ (that is, the first-order class of frames defined by $\Sigma$)''.\parencite[The Sahlqvist Completeness Theorem; Theorem 4.42, p.~212]{Blackburn2001},
\end{quote}

By setting $\Sigma = \{\mathbf{D}, \mathbf{C}\}$, it follows immediately that \textbf{KDC} is sound and strongly complete with respect to the class of serial, confluent Kripke frames.

This frame-theoretic behavior signals a foundational departure from the standard G\"odel-L\"ob (\textbf{GL}) paradigm. While \textbf{GL} is bounded by irreflexivity and 
permits pathological models where a theory asserts its own inconsistency, \textbf{KDC} considers 
the single-world reflexive frame ($wRw$) as an ideal and as a valid semantic baseline. In this truth-directed system, the confluence axiom \textbf{C} acts as a geometric filter: rather than forcing an open-ended extensional ascent to ever-stronger meta-theories, 

Unlike structurally flat systems such as $\mathbf{S5}$, $\mathbf{KDC}$ has a desirable modal asymmetry for the notion of provability. While it--- given its Axiom \textbf{C} and theorem that $\vdash\Box\Box A\to\Box A$---is fully reductive along its $\Box$ axis---and so collapses iterative proof statements into non-iterative ones, it is non-reductive along its $\Diamond$ axis. Combined with the confluence imposed upon KDC-frames by Axiom \textbf{C}, this preserves the delicate and desired structural hierarchy of iterative consistency statements, or $\Diamond$ theses, fixing \textbf{KDC} as a truth-directed provability logic.

We have in this section focused entirely upon a modal logical discussion of the provability theory \textbf{KDC}. This will be rectified in the next section.

\subsection{\textbf{KDC} is the provability theory of \textbf{VT}}

Visibility theory does not use operators apart from the logical ones. Instead it has predicates on Bend codes, as $\Visible$. 

Our discussion of \textbf{KDC} in the previous subsection was driven by the fact that it is common to use modal logics in the logic of provability literature. And that practice is useful in many contexts. But I take such uses of $\Box$ and $\Diamond$ to be imitations, and in the last analysis inadequate. The intention now is that provability theory \textbf{KDC} is formulated canonically in this section.

\textbf{VT} has
the syntactic provability predicate $\spadesuit$, and its dual $\clubsuit$. These are predicates of the Bend coded formulas, and their behavior is directed by axiomatic fiat in \textbf{VT}. So there is not a detour via arithmetic at this point. For we are not now aiming at proving the incompleteness of arithmetic. We rather aim at proving it for \text{VT}. 
In accordance with the foregoing discussion, the following shall hold for $\spadesuit$ and $\clubsuit$ in \textbf{KDC} -- and by extension \textbf{VT}.

\medskip
Provability theory \textbf{KDC}:

\begin{enumerate}[label=\textbf{(KDC\arabic*)}, leftmargin=*, labelindent=3em]
\setlength{\itemsep}{0.6em}

\item\label{dfn:KDCdfn} $\ \ \ \clubsuit\Bendcode{A}\coloneqq\lnot\spadesuit\Bendcode{\lnot A}.$ 

\item \label{ax:KDCK}
$\vdash \spadesuit\Bendcode{A \to B}\to (\spadesuit\Bendcode{A}\to\spadesuit\Bendcode{B}) $.

\item \label{ax:KDCD}
$\vdash \spadesuit\Bendcode{A}\to\clubsuit\Bendcode{A}$

\item \label{ax:KDCC}
$\vdash \spadesuit\Bendcode{\clubsuit\Bendcode{A}}\to\spadesuit\Bendcode{A}$

\item \label{rule:KDCR}
$\vdash A \ \Rightarrow \ \ \vdash \spadesuit\Bendcode{A}$.

\end{enumerate}

\subsection{The incompleteness of \textbf{VT}}

\begin{thm}
There is a sentence
which $\mathbf{VT}$ neither proves nor refutes.
\end{thm}

\begin{proof}
    Given Theorem \ref{presentationaldiagonallemma}, there is a sentence $U$ such that

\begin{equation}\label{fixpointnotspade}
     \vdash U\leftrightarrow\lnot\spadesuit\Bendcode{U}.
\end{equation}

Assume $\vdash U$. By inference rule \ref{rule:KDCR}, $\vdash \spadesuit \Bendcode{U}$. Given Equation \eqref{fixpointnotspade},  $\vdash\lnot U$. So it follows that $\not\vdash U$. 

Suppose $\vdash\lnot U$. By inference rule \ref{rule:KDCR}, $\vdash\spadesuit\Bendcode{\lnot U}$. By Axiom \ref{ax:KDCD} and Definition \ref{dfn:KDCdfn}, $\vdash\lnot\spadesuit\Bendcode{U}$. Inconsistency! So $\not\vdash \lnot U$.
\end{proof}

\newpage

\section{Identity}\label{identitysection}

The aim of the following constructions is to obtain an adequate,
fully classical theory of identity from the minimum of axiomatic
resources necessary beyond those already present in $\mathbf{VT}$.

Whitehead, who was Quine's doctoral advisor, according to \parencite{Forster2019}  
advised Quine that $\{y \mid x \in y\}$ should be called the \emph{essence} of $x$.

\subsection{Co-essentiality}
We define the identity relation by means of co-essentiality, akin to the relation called membership congruency by Fraenkel and Bar-Hillel \parencite[27]{Fraenkel1973}, though not used in \parencite{Fraenkel1958}.
  \begin{dfn}{~}\label{coessentialdefinition}        
		\[\textnormal{Visions} \ a \ \textnormal{and} \ b \ \textnormal{are \emph{co-essential} just if} \ \forall u(a\in u\to b\in u).\]

        \EndDef

  \end{dfn}

  \begin{dfn}[(Identity as co-essentiality)]\label{identitydefinition} 
			\[a= b \ \coloneqq \forall u(a\in u\to b\in u)\]
            \EndDef
		\end{dfn}

        The justification given for the analogous set theoretic definition \(\ast13\!\cdot\!01\) in \parencite[175]{whitehead_russell_1910} will not  justify  Definition \ref{identitydefinition} in $\mathbf{VT}$. The symmetry  of the definiens in Definition  \ref{identitydefinition} is in $\mathbf{VT}$ a consequence of 
        Theorem \ref{SYMMETRY} below, and the proof of the latter  does not, as  the Principia Mathematica justification of \(\ast13\!\cdot\!01\), appeal to   \emph{predicativity} or  
        anything like an
        \emph{Axiom of reducibility}. 
		
\noindent

\subsection{Visibility and persistence of identity}

\begin{axiom}[on the visibility of inequality]\label{inequalityaxiom}
    \[\vdash \Visible\Bendcode{\forall x\forall y(x\neq y\to \Visible\Bendcode{x\neq y})}.\]
\end{axiom}
\begin{remark}[on Axiom \ref{inequalityaxiom}]
    Given Postulate \ref{semanticcorrelation}, 
\end{remark}

\begin{thm}[on the visibility of identity]\label{identitypersistence}
    \[\vdash \Visible\Bendcode{\forall x\forall y(x= y\to \Visible\Bendcode{x=y})}.\]
\end{thm}

\begin{proof}
   From Definition \ref{identitydefinition} we have that 
   \[
   \vdash a=b\to (a\in\Anschauung{x}{a=x}\to b\in\Anschauung{x}{a=x}).
   \]

   $\vdash a\in\Anschauung{x}{x=a}$ holds as by \textbf{VCA}:~p.~\pageref{ax:VCA}  and the fact that $\vdash\Visible\Bendcode{a=a}$. Therefore, 

    \[
   \vdash a=b\to b\in\Anschauung{x}{a=x}).
   \]

   By \textbf{VCA}:~p.~\pageref{ax:VCA} on the consequent of the formula in the last step,

   \[
   \vdash a=b\to \Visible\Bendcode{a=b}.
   \]

Generalization and the reasoning's visibility justify the target equation.
   
\end{proof}

\begin{thm}[on identity and inequality persistence.]\leavevmode\par\vspace{2mm}\label{idinpersistence}

If for some ordinal $\beta$, $\vDash^{\beta} a=b$, then for \emph{any} ordinal $\gamma$, $\vDash^{\gamma} a=b$.  

\smallskip

If for some ordinal $\beta$, $\vDash^{\beta} c\neq d$, then for \emph{any} ordinal $\gamma$, $\vDash^{\gamma} c\neq d$.       
\end{thm}

\begin{proof}
    Given Postulate \ref{semanticcorrelation} and Axiom \ref{inequalityaxiom}, it follows that for \emph{any} ordinal $\gamma$, 
 
    \begin{equation}
    \vDash^{\gamma} \Visible\Bendcode{\forall x\forall y(x\neq y\to \Visible\Bendcode{x\neq y})}.
        \end{equation}

The proof of Theorem \ref{identitypersistence} reveals that only axiomatic principles are needed in it, so by Postulate \ref{semanticcorrelation}, 

    \begin{equation}\label{identityVisEq}
    \vDash^{\gamma} \Visible\Bendcode{\forall x\forall y(x=y\to \Visible\Bendcode{x=y})}.
        \end{equation}

It is easy to show that the next two equations follow.

 \begin{equation}\label{inequalityVis}
    \vDash^{\gamma} \Visible\Bendcode{\forall x\forall y(x\neq y\leftrightarrow \Visible\Bendcode{x\neq y})}.
        \end{equation}

         \begin{equation}\label{identityVis}
    \vDash^{\gamma} \Visible\Bendcode{\forall x\forall y(x=y\leftrightarrow \Visible\Bendcode{x=y})}.
        \end{equation}

        In any revision sequence, $\vDash^0 a=b$ or $\vDash a\neq b$, according to what is chosen for its maximal consistent set of formulas. For the semantics of the revision semantics, as stated in \S~\ref{revisionsemantics}, only requires the fulfillment of the condition $\Sigma \gamma\prec\beta\Pi \delta(\gamma\leq \delta<\beta\Rightarrow \ \vDash^{\delta}A)$ if $0\prec\beta$, to validate $\vDash^{\beta} A$. If $\vDash^0 a=b$, $\vDash^{\gamma} a=b$ for \emph{any} ordinal $\gamma$ because Equation \ref{identityVis} is true. The analogous argument establishes that $\vDash^{\gamma} a\neq b$ for \emph{any} ordinal if $\vDash^0 a\neq b$.

\end{proof}

\subsection{The adequacy of identity as co-essentiality}\label{identityadequacy}

\begin{thm}[\textbf{Orthodoxy, equivalence and fungibility of identity}:]\label{Identitytheorem}
			\begin{enumerate}[itemindent =0.4cm, before=\leavevmode, label=\((\arabic*)\),ref=  \ref{Identitytheorem}.\(\arabic*\)]
				\item\label{Id1}\,\, \(\vdash\ \Visible\Bendcode{\Visible\Bendcode{a= b}\vee \Visible\Bendcode{a\neq b}}\)\hfill Orthodoxy
				\item \,\, \(\vdash \Visible\Bendcode{a= a}.\) \hfill Reflexivity
				\item \,\, \(\vdash \Visible\Bendcode{a= b\wedge b= c\to a= c}.\)\hfill Transitivity
				\item\label{SYMMETRY} \,\, \(\vdash\Visible\Bendcode{a= b\to b= a}.\)\hfill Symmetry
				\item\label{fungibility} \,\,  \(\vdash  \Visible\Bendcode{a= b\to (A^a_v\to A^b_v)}.\) \hfill Fungibility
			\end{enumerate}
		\end{thm}
		
		\begin{proof} \mbox{}
			\begin{enumerate}
				\item Use the fact that $\vdash a=b\vee a\neq b$, Axiom \ref{inequalityaxiom} and Theorem \ref{identitypersistence} plus  a disjunctive syllogism.  
              
                \item Trivial
				\item Trivial, given Definition \ref{identitydefinition}
				\item  

Clearly,
\begin{equation}
        \vdash a=b\to (a\in \Anschauung{w}{w=a}\to b\in \Anschauung{w}{w=a}).
\end{equation}

But 
\begin{equation}
    \vdash a\in \Anschauung{w}{w=a},
\end{equation}

so that by visible comprehension,

\begin{equation}\label{alethiccomprehensionuse}
\vdash a=b\to  \Visible\Bendcode{b=a}. 
\end{equation}

On account of Axiom \ref{inequalityaxiom} combined with the unvisibilized Axiom \ref{ax:V5}, 

\begin{equation}\label{Tbisatobisa}
\vdash\Visible\Bendcode{b=a}\to   b=a
\end{equation}

so by Equations \eqref{alethiccomprehensionuse} and \eqref{Tbisatobisa} and transitivity of entailment

\begin{equation}\label{aisbtobisa}
\vdash a=b\to   b=a. 
\end{equation}

\item

By definition $\vdash a=b\to \forall u(a\in u\to b\in u).$ So clearly 
\[\vdash\forall u(a\in u\to b\in u)\to (a\in\Anschauung{u}{A^a_v\to A^u_v}\to b\in\Anschauung{u}{A^a_v\to A^u_v}),\] where $A^a_v$ is the result of substituting $a$ for $v$ in $A$, and
$A^u_v$ is the result of substituting $u$ for $v$ in $A$.
As clearly $\vdash a\in\Anschauung{u}{A^a_v\to A^u_v}$, it follows that $\vdash a=b\to b\in\Anschauung{u}{A^a_v\to A^u_v}.$ Given the Axiom of Visible Comprehension, Axiom \ref{ax:VCA}, it follows that $\vdash a=b\to \Visible\Bendcode{A^a_v\to A^b_v}.$ So on account of Axiom \ref{inequalityaxiom} combined with the unvisibilized Axiom \ref{ax:V5}, $\vdash a=b\to (A^a_v\to A^b_v).$
\end{enumerate}
\end{proof}

\newpage

\section{Arithmetic}\label{Arithmetic}

\subsection{On order and strength}
Given Theorem \ref{Arithmetictheorem}, the arithmetic with visions of \S\ref{Arithmetic} is second-order Peano arithmetic. According to \parencite{Simpson2013} it is categorical. 

The system $\mathbf{Z}_2$, named Second-Order Arithmetic, has the same strength as $\mathbf{ZF}$ minus the power set axiom. One should not confuse second-order Peano arithmetic, known as  $\mathbf{PA}_2$, with the much stronger system $\mathbf{Z}_2$.
We call Theorem \ref{secondorder} \emph{vision-induction}, and
Theorem \ref{schemainduction} \emph{schema-induction}. The combination of vision-induction and schema-induction places the resulting arithmetic in close proof-theoretic proximity to subsystems of second-order arithmetic such as $\mathbf{ACA}$, with proof-theoretic ordinal $\epsilon_{\epsilon_0}$. 
$\mathbf{ACA}$ lies between $\mathbf{ACA}_0$ and the much stronger system $\mathbf{ATR}_0$.

The proofs below are more complicated than usual, as it is necessary to obtain the visibilized versions of the arithmetical axioms. To distinguish from $\varnothing$ and $\mathbb{N}$, $V_\varnothing$ and $V_{\mathbb N}$ are used
for the empty vision and the vision of natural numbers which are visions, respectively.

\subsection{The formal system}

    \begin{dfn}\label{emptysuccessornaturals}
		\begin{enumerate}[itemindent =0.4cm, before=\leavevmode, label=$(\arabic*)$,ref=  \ref{emptysuccessornaturals}.$\arabic*$]
			\item \(V_{\varnothing}=\Anschauung{x}{x\neq x}.\)
			\item\label{successor} \(a'=\Anschauung{x}{x=a\vee x\in a}.\)
			\item\label{omegadefinition} \(V_{\mathbb N}=\Anschauung{x}{\forall y(V_{\varnothing}\in y\wedge\forall z(z\in y\to z'\in y)\to x\in y)}.\)
		\end{enumerate}
        \EndDef
	\end{dfn}
	
	\begin{thm}\label{Arithmetictheorem}
		\begin{enumerate}[itemindent =0.4cm, before=\leavevmode, label=$(\arabic*)$,ref=  \ref{Arithmetictheorem}.$\arabic*$]
			\item\label{0inomega} \(\vdash\Visible\Bendcode{V_{\varnothing}\in V_{\mathbb N}}.\)
			\item\label{xinomegatoxsharpinomega} \(\vdash\Visible\Bendcode{\forall x(x\in V_{\mathbb N}\to x'\in V_{\mathbb N}}.\)
            \item\label{0notsuccessor} \(\vdash\Visible\Bendcode{\lnot \exists x(V_{\varnothing}=x')}.\)
			\item\label{omegaisorthodox} \(\vdash\Visible\Bendcode{\forall x(\Visible \Bendcode{x\in V_{\mathbb N}}
        \vee\Visible\Bendcode{x\notin V_{\mathbb N}}})\), i.e. \(V_{\mathbb N}\) is orthodox.
			\item\label{secondorder} \(\vdash\Visible\Bendcode{\forall y(V_{\varnothing}\in y\wedge\forall z(z\in y\to z'\in y)\to \forall x(x\in V_{\mathbb N}\to x\in y))}. \)
			\item\label{schemainduction} \(\vdash\Visible\Bendcode{\mathrm{A}(V_{\varnothing})\wedge\forall x(\mathrm{A}(x)\to \mathrm{A}(x'))\to \forall y(y\in V_{\mathbb N}\to \mathrm{A}(y))}.\)
            \item $\vdash \Visible\Bendcode{\forall y\forall z\bigl(y\in V_{\mathbb N} \wedge \ z\in V_{\mathbb N}\to (y'=z'\to y=z)\bigr)}.$
		\end{enumerate}
	\end{thm}

    \begin{proof}[Proof:] \mbox{}
		
		\setcounter{equation}{0}
		\begin{enumerate}[itemindent =0.4cm]
			\item Combine visible comprehension and the fact that 
            \begin{equation}\label{varnothingtriviality}
            \vdash\Visible\Bendcode{\forall y(V_{\varnothing}\in y\wedge\forall z(z\in y\to z'\in y)\to V_{\varnothing}\in y)}    
            \end{equation}
            
			\item This follows from visible comprehension and the evident 

            \begin{equation}\label{omegaxtoomegaxpluss}
            \begin{split}
                \forall x\Bigl(&\Visible\Bendcode{\forall y(V_{\varnothing}\in y\wedge \forall z(z\in y\to z'\in y)\to x\in y)}\to\\
                &\Visible\Bendcode{\forall y(V_{\varnothing}\in y\wedge \forall z(z\in y\to z'\in y)\to x'\in y)}\Bigr)
                \end{split}
            \end{equation}

            \item Given Definition \ref{successor}, $\forall x(x\in x')$. So $\forall x(V_{\varnothing}=x'\to x\in V_{\varnothing})$. Consequently, if $V_{\varnothing}$ were identical to a vision $a'$ then $a$ would be a member of $V_{\varnothing}$, contrary to the definition of $V_{\varnothing}$.

			\item From logic: 
         \begin{equation}\label{logicprecursor}
         \begin{split}
        \vdash & V_\varnothing\in V_\mathbb{N}\wedge\forall x(x\in V_\mathbb{N}\to x'\in V_\mathbb{N})\to\\ & (\forall y(V_\varnothing\in y\wedge\forall x(x\in y\to x'\in y)\to a\in y)\to a\in V_\mathbb{N}).
        \end{split}
     \end{equation}
     Equation \eqref{logicprecursor} plus Theorems \ref{0inomega} and \ref{xinomegatoxsharpinomega} entail that 

     \begin{equation}
          \vdash\forall y(V_\varnothing\in y\wedge\forall x(x\in y\to x'\in y)\to a\in y)\to a\in V_\mathbb{N}
     \end{equation}

     The reasoning in these steps has only used principles that are visible in \textbf{VT}, so

\begin{equation}\label{visiblesteptoVN}
          \vdash\Visible\Bendcode{\forall y(V_\varnothing\in y\wedge\forall x(x\in y\to x'\in y)\to a\in y)\to a\in V_\mathbb{N}}
     \end{equation}
     
Apply next Axiom \ref{ax:V4} appropriately to Equation \ref{visiblesteptoVN} to get

\begin{equation}\label{visvisiblesteptoN}
\begin{split}
        \vdash\Visible\Bendcode{\Visible\Bendcode{\forall y(V_\varnothing\in y\wedge\forall x(x\in y\to x'\in y)\to a\in y)}\to \Visible\Bendcode{a\in V_\mathbb{N}}}
        \end{split}
\end{equation}

Apply next Axiom \textbf{VCA}:~p.~\pageref{ax:VCA} to the antecedent in Equation \ref{visvisiblesteptoN} to get

\begin{equation}\label{ainVntoVisainVn}
    \vdash\Visible\Bendcode{a\in V_\mathbb{N}\to \Visible\Bendcode{a\in V_\mathbb{N}}}
\end{equation}

Combine Equation \ref{ainVntoVisainVn} with Axiom \ref{ax:V5} to get

\begin{equation}
    \vdash\Visible\Bendcode{a\in V_\mathbb{N}\to \lnot \Visible\Bendcode{a\notin V_\mathbb{N}}}
\end{equation}

          Contrapose internally to get

    \begin{equation}
    \vdash\Visible\Bendcode{\Visible\Bendcode{a\notin V_\mathbb{N}}\to a\notin V_\mathbb{N}}
\end{equation}  

Next appeal to the instance 
\[
\vdash\Visible\Bendcode{\Visible\Bendcode{a\notin V_\mathbb{N}}\to a\notin V_\mathbb{N}}\to \Visible\Bendcode{a\in V_\mathbb{N}\vee a\notin V_\mathbb{N}}
\]

of Axiom \ref{ax:V6}, and visibilize to get

\begin{equation}
\vdash\Visible\Bendcode{a\in V_\mathbb{N}\vee a\notin V_\mathbb{N}}
\end{equation}

Finally by logic, as $a$ was arbitrary,

\begin{equation}
\vdash\Visible\Bendcode{\forall x(x\in V_\mathbb{N}\vee x\notin V_\mathbb{N})}
\end{equation}

So by Definition \ref{orthodoxydefinition}, $V_\mathbb{N}$ is orthodox. 

			\item This is logically equivalent to \[\vdash\Visible\Bendcode{\forall x(x\in V_{\mathbb N}\to\forall y(V_{\varnothing}\in y\wedge\forall z(z\in y\to z'\in y)\to  x\in y))}. \]

            By Axiom \textbf{VCA}:~p.~\pageref{ax:VCA}, 

            \[\vdash\Visible\Bendcode{\forall x(x\in V_{\mathbb N}\leftrightarrow\Visible\Bendcode{\forall y(V_{\varnothing}\in y\wedge\forall z(z\in y\to z'\in y)\to  x\in y))}}. \]

            But, as $V_{\mathbb{N}}$ is orthodox, this simplifies to 

            \[\vdash\Visible\Bendcode{\forall x(x\in V_{\mathbb N}\leftrightarrow\forall y(V_{\varnothing}\in y\wedge\forall z(z\in y\to z'\in y)\to  x\in y))}, \]

which immediately yields the desired conclusion:
            \[\vdash\Visible\Bendcode{\forall x(x\in V_{\mathbb N}\to\forall y(V_{\varnothing}\in y\wedge\forall z(z\in y\to z'\in y)\to  x\in y))}. \]
			
			\item We presuppose 
\begin{dfn}\label{defAstar}
$A^*(x)\coloneqq
(A(V_{\varnothing})\wedge
\forall x(A(x)\to A(x')))
\to A(x).$    
\end{dfn}

            By logic,
           \begin{equation}\label{Astarempty}
               A^*(V_{\varnothing}).
                       \end{equation}

                       By logic,
           \begin{equation}\label{Astarsuccessor}
               \forall x(A^*(x)\to A^*(x')).
                       \end{equation}

Since Equations \ref{Astarempty} and \ref{Astarsuccessor} only depend upon logic, $\forall x(A^*(x)\to A^*(x'))$ and $A^*(V_{\varnothing})$ are \emph{visibly} derived. Consequently,

\begin{equation}\label{VisAstarempty}
               \vdash\Visible\Bendcode{A^*(V_{\varnothing})}.
                       \end{equation}

                       and
\begin{equation}\label{VisAstarsuccessor}
               \vdash\Visible\Bendcode{\forall x(A^*(x)\to A^*(x'))}.
                       \end{equation}

                       By the fact related in Footnote \ref{barcanfootnote}:~p.~\pageref{barcanfootnote}, Equation \ref{VisAstarsuccessor} entails

\begin{equation}\label{VisAstarsuccessorQuantifiershifted}
               \vdash\forall x\Visible\Bendcode{A^*(x)\to A^*(x')}.
                       \end{equation}

                       An internal use of the Reflexive Axiom \ref{ax:IA1}:~p.~\pageref{ax:IA1}, Axiom \ref{ax:V4}:~p.~\pageref{ax:V4} and Equation \ref{VisAstarsuccessorQuantifiershifted}:~p.~\pageref{VisAstarsuccessorQuantifiershifted} entail

\begin{equation}\label{VisAstarsuccessorQuantifiershiftedVisin}
               \vdash\forall x(\Visible\Bendcode{A^*(x)}\to \Visible\Bendcode{A^*(x')}).
                       \end{equation}

                       Visible comprehension \textbf{VCA}:~p.~\pageref{ax:VCA} and Equation \eqref{VisAstarsuccessorQuantifiershiftedVisin}:~p.~\pageref{VisAstarsuccessorQuantifiershiftedVisin} entail

                      \begin{equation}\label{VisAstarsuccessorQuantifiershiftedVISIONARY}
               \vdash\forall x(x\in\Anschauung{x}{A^*(x)}\to x'\in\Anschauung{x}{A^*(x)}).
               \end{equation}

               Equation \ref{VisAstarempty}:~p.~\pageref{VisAstarempty} and visible comprehension \textbf{VCA}:~p.~\pageref{ax:VCA} entail

               \begin{equation}\label{Vsubvisemptyin}
                   \vdash V_{\varnothing}\in\Anschauung{x}{A^*(x)}.
               \end{equation}

Combining Equations \eqref{VisAstarsuccessorQuantifiershiftedVISIONARY} and \eqref{Vsubvisemptyin} yields

\begin{equation}\label{antecedentforschemainduction}
\begin{split}
    \vdash V_{\varnothing}\in\Anschauung{x}{A^*(x)}\wedge \forall x(x\in\Anschauung{x}{A^*(x)}\to x'\in\Anschauung{x}{A^*(x)}).
    \end{split}
\end{equation}

It is routine to verify that Equation \eqref{antecedentforschemainduction}:~p.~\pageref{antecedentforschemainduction} was visibly derived. This justifies 

\begin{equation}\label{VISantecedentforschemainduction}
\begin{split}
    \vdash \Visible\Bendcode{V_{\varnothing}\in\Anschauung{x}{A^*(x)}\wedge \forall x(x\in\Anschauung{x}{A^*(x)}\to x'\in\Anschauung{x}{A^*(x)})}.
    \end{split}
\end{equation}

               Axioms \ref{ax:IA1}:~p.~\pageref{ax:IA1} and \ref{ax:V4}:~p.~\pageref{ax:V4} combined with Equation \eqref{VISantecedentforschemainduction} and Theorem \ref{secondorder}:~p.~\pageref{secondorder} entail 

               \begin{equation}\label{universalpenultimatestep}
    \vdash\Visible\Bendcode{\forall y(V_{\varnothing}\in y\wedge\forall z(z\in y\to z'\in y)\to \forall x(x\in V_{\mathbb N}\to x\in y))}.
\end{equation}

By internal universal instantiation in Equation \ref{universalpenultimatestep}:~p.~\pageref{universalpenultimatestep} we get

\begin{equation}\label{instantiatedpenultimatestep}
  \vdash\Visible\Bendcode{\forall x(x\in V_{\mathbb N}\to x\in\Anschauung{x}{A^*(x)})}.  
\end{equation}

As $V_{\mathbb N}$ is orthodox by Theorem \ref{omegaisorthodox}:~p.~\pageref{omegaisorthodox}, 

\begin{equation}
     \vdash\Visible\Bendcode{\forall x(\Visible\Bendcode{x\in V_{\mathbb N}}\to x\in V_{\mathbb N})},
\end{equation}

so it follows that

\begin{equation}\label{VsubNstep}
  \vdash\Visible\Bendcode{\forall x(\Visible\Bendcode{x\in V_{\mathbb N}}\to x\in\Anschauung{x}{A^*(x)})}.  
\end{equation}

Equation \ref{VsubNstep} and visible comprehension \textbf{VCA}:~p.~\pageref{ax:VCA} entail

\begin{equation}\label{VissubNstep}
  \vdash\Visible\Bendcode{\forall x(\Visible\Bendcode{x\in V_{\mathbb N}}\to \Visible\Bendcode{A^*(x)})}.  
\end{equation}

From Equation \ref{VissubNstep}:~p.~\pageref{VissubNstep} and Footnote \ref{barcanfootnote}:~p.~\pageref{barcanfootnote} we obtain

\begin{equation}\label{quantifiershiftedVissubNstep}
  \vdash\forall x\Visible\Bendcode{(\Visible\Bendcode{x\in V_{\mathbb N}}\to \Visible\Bendcode{A^*(x)})}.  
\end{equation}

Equation \ref{quantifiershiftedVissubNstep}:~p.~\pageref{quantifiershiftedVissubNstep}  and Planer's schema \ref{planer}:~p.~\pageref{planer} entail

\begin{equation}\label{flatquantifiershiftedVissubNstep}
  \vdash\forall x\Visible\Bendcode{(x\in V_{\mathbb N}\to A^*(x))}.  
\end{equation}

An appeal to Axiom \ref{ax:IA2} justifies

\begin{equation}\label{flatquantifierPrefixedVissubNstep}
  \vdash\Visible\Bendcode{\forall x(x\in V_{\mathbb N}\to A^*(x))}.  
\end{equation}

An appeal to Definition \ref{defAstar}:~p~\pageref{defAstar} justifies

\begin{equation}\label{VsubNandAstarspealledout}
    \vdash\Visible\Bendcode{\forall x(x\in V_{\mathbb N}\to((A(V_{\varnothing})\wedge
\forall x(A(x)\to A(x')))
\to A(x)))}, 
\end{equation}

and Equation \ref{VsubNandAstarspealledout} is equivalent to

  \begin{equation}\label{rearrangedVsubNandAstarspealledout}
    \vdash\Visible\Bendcode{(A(V_{\varnothing})\wedge
\forall x(A(x)\to A(x'))\to
\forall x(x\in V_{\mathbb N}\to A(x)))}, 
\end{equation}

\item The base case  is $\vdash\bigwedge_{k=0}^{k=0} \bigwedge_{l=0}^{l=0}(k\neq l\to k'\neq l')$, which  clearly holds.

It is  as well obvious that the following induction step holds:

\begin{equation}
  \vdash\forall x\Biggl(\bigwedge_{k=0}^{k=x} \bigwedge_{l=0}^{l=x}(k\neq l\to k'\neq l')\to\bigwedge_{k=0}^{k=x'} \bigwedge_{l=0}^{l=x'}(k\neq l\to k'\neq l')\Biggr)  
\end{equation}

\bigskip

So on account of  Theorem \ref{schemainduction} it follows that   

\begin{equation}\label{stringinductionconsequent}
    \vdash\forall x(x\in V_{\mathbb N}\to \bigwedge_{k=0}^{k=x} \bigwedge_{l=0}^{l=x}(k\neq l\to k'\neq l')).
\end{equation}

Consider the instances 

\begin{equation}
    \vdash y\in V_{\mathbb N}\to \bigwedge_{k=0}^{k=y} \bigwedge_{l=0}^{l=y}(k\neq l\to k'\neq l')
\end{equation}

and 
\begin{equation}
    \vdash z\in V_{\mathbb N}\to \bigwedge_{k=0}^{k=z} \bigwedge_{l=0}^{l=z}(k\neq l\to k'\neq l'),
\end{equation}

while presupposing that 
\begin{equation}\label{ylargerthanz}
    y>z.
\end{equation}

\bigskip

Clearly: 

\begin{equation}
 \bigwedge_{k=0}^{k=y} \bigwedge_{l=0}^{l=y}(k\neq l\to k'\neq l') \to \bigwedge_{l=0}^{l=y}(y\neq l\to y'\neq l').
\end{equation}

\bigskip

Given Equation \eqref{ylargerthanz},

\begin{equation}    
\bigwedge_{l=0}^{l=y}(y\neq l\to y'\neq l'\to \bigwedge_{l=0}^{l=z}(y\neq l\to y'\neq l'.
\end{equation}

\smallskip

But

\smallskip

\begin{equation}
    \bigwedge_{l=0}^{l=z}(y\neq l\to y'\neq l')\to (y\neq z\to y'\neq z'),
\end{equation}

\smallskip

so that  

\smallskip

\begin{equation}\label{penultimateinductionconsequent}
\begin{split}
        \forall x\Biggl(x\in V_{\mathbb N}\to &\bigwedge_{k=0}^{k=x} \bigwedge_{l=0}^{l=x}(k\neq l\to k'\neq l')\Biggr)\to\\ 
        &\forall y\forall z\Biggl(y\in V_{\mathbb N} \wedge \ z\in V_{\mathbb N}\to (y\neq z\to y'\neq z')\Biggr).
\end{split}
 \end{equation}

Modus ponens with Equations \eqref{stringinductionconsequent} and \eqref{penultimateinductionconsequent}, and contraposing  the result's consequent entail

\begin{equation}
    \vdash\forall y\forall z\bigl(y\in V_{\mathbb N} \wedge \ z\in V_{\mathbb N}\to (y'=z'\to y=z)\bigr).
\end{equation}

Observe finally that 

\begin{equation}
    \vdash\Visible\Bendcode{\forall y\forall z\bigl(y\in V_{\mathbb N} \wedge \ z\in V_{\mathbb N}\to (y'=z'\to y=z)\bigr)}.
\end{equation}

because $V_{\mathbb N}$ and identity are orthodox.
             \end{enumerate}
	\end{proof}

\newpage
\section{Visible Barcan failure:~the invisibility of \textbf{IA1}}\label{secBarcanfailure}\leavevmode\par\vspace{2mm}

	\noindent The negative result of this section is that the Barcan formula for $\Visible$ cannot be adopted visibly, as in $\vdash\Visible\Bendcode{\forall x\Visible\Bendcode{A(x)}\to \Visible\Bendcode{\forall x A(x)}}$. We show that this result follows from a negative result for a truth theoretic context, which was obtained in \parencite{McGee1985}\index{McGee, Vann Roger}, and is known as \emph{McGee's paradox}. Below we adapt from McGee's argument to our theory of visions.  Compare the account in \parencite[380-382]{Cantini1996}\index{Cantini, Andrea} and the observation in \parencite[357]{Bjordal2012}.
	
	We first introduce some vision-theoretic versions 
    of set theoretic notions. That is necessary, for we do not have sets yet.
    
\begin{dfn}[of Russell's vision]
\begin{equation*}
    \mathfrak{r}=\Anschauung{x}{x\notin x}
\end{equation*}
\end{dfn}

\begin{dfn}\label{Barcanfailure}
		
		\begin{enumerate}[itemindent =0.3cm, before=\leavevmode, label=(\arabic*),ref=   \ref{Barcanfailure}.\arabic*]

        \item\label{visionsofemptyandnatural} As in Definition \ref{successor}:~p.~\pageref{successor}, $V_{\varnothing}$ is the empty vision $\Anschauung{x}{x\neq x}$, and $V_{\mathbb N}$ is the vision of the natural numbers, construed as visions. But here $\omega$ is used in connection with $\omega$-ordered visions of visions which are not vision-numbers.
            \item\label{BFfsuccessor} As in Definition \ref{successor}:~p.~\pageref{successor}, \(a'=\Anschauung{x}{x\in a\vee x=a}\).
			\item  \(\Proposisjon{a, b}=\Anschauung{x}{x=a\vee x=b}\).
			\item \(\Proposisjon{a}=\Proposisjon{a, a}\).
            \item\label{aomega}
\begin{multline*}
a_\omega=
\Anschauung{u}{\forall x(
\langle\,V_{\varnothing},a\rangle\in x
\\
\wedge
\forall y,z(
\langle y,z\rangle\in x
\to
\langle y',\Anschauung{v}{v\in z}\rangle
)
\to u\in x)}.
\end{multline*}            
			
			\item\label{onlyatlimits} \(\mathrm{t}=\Anschauung{x}{x=\mathfrak{r}\wedge x\notin x\wedge\lnot\Visible\Bendcode{x\in x}}\).
			
			\item Use \(\overline{0}\), \(\overline{1}\), \(\overline{2}\), \ldots for the members  of \(\omega\).
			\item  Let \(\mathrm{t}_{\overline{0}}= \mathrm{t}\) and
			\(\mathrm{t}_{\overline{n+1}}= \Anschauung{v}{v\in \mathrm{t}_{\overline{n}}}. \) 
			\item\label{Carcanfailurestatement} \(\mathrm{B}(\mathrm{t}_{\overline{i}})= \exists w(\langle w,\mathrm{t}_{\overline{i}}\rangle\in \mathrm{t}_{\omega})\to \mathfrak{r}\notin \mathrm{t}_{\overline{i}}.\)
			
			\item\label{CCarcanfailurestatement} \(\mathrm{B}(x)= \exists w(\langle w,x\rangle\in \mathrm{t}_{\omega})\to \mathfrak{r}\notin x.\)
			
		\end{enumerate}
		\end{dfn}
		\medskip

	\begin{lem}\label{OrrthodoxyoftsubN}
 \[\textnormal{For any} \ a, a_\omega \ \textnormal{is orthodox.}\]
	\end{lem}
	\begin{proof}
		Adapt the proof of Theorem \ref{omegaisorthodox}.\end{proof}
	
	\begin{lem}{~}\label{russelatlimit}
\[\vDash^\lambda \mathfrak{r}=\mathfrak{r}\wedge \mathfrak{r}\notin\mathfrak{r}\wedge\lnot\Visible\Bendcode{\mathfrak{r}\in\mathfrak{r}} \ \textnormal{just if} \ \lambda \ \textnormal{is a limit.}\] 
	\end{lem}
	\begin{proof}
		For any  successor ordinal \(\chi+1\), \(\vDash^{\chi+1} \lnot\Visible\Bendcode{\mathfrak{r}\in\mathfrak{r}}\leftrightarrow \mathfrak{r}\in\mathfrak{r}.\)  So the Lemma holds as at precisely   limit ordinals \(\lambda\), \(\vDash^\lambda \mathfrak{r}\notin\mathfrak{r}\wedge \lnot \Visible\Bendcode{\mathfrak{r}\in\mathfrak{r}}. \)
	\end{proof}

	\begin{lem}[~]\label{failureofthebarcanformula}  
    \mbox{}
    
    For any limit ordinal \(\alpha\) preceding closure ordinal \(\varpi\), and \(\beta=\alpha+\omega:\)

	\begin{enumerate}
		\item 	\( \vDash^\beta \forall x\Visible\Bendcode{\mathrm{B}(x)}\)	
		\item \( \vDash^\beta\lnot\Visible\Bendcode{\forall x\mathrm{B}(x)}.\)
	\end{enumerate}
\end{lem}

\begin{proof} 

\mbox{}

	\begin{enumerate}
		\item	If \(\vDash^\beta\lnot\exists w(\langle w,x\rangle\in \mathrm{t}_{\omega})\), it follows that  \(\vDash^\beta\Visible\Bendcode{\mathrm{B}}(x)\)  on account of 
		Lemma \ref{OrrthodoxyoftsubN}.   
		If, on the other hand,  \(\vDash^\beta\exists w(\langle w, x\rangle\in \mathrm{t}_{\omega})\) we have  that  \(\vDash^\beta \Visible\Bendcode{\mathrm{B}}(x)\), as there is a \(\gamma \succeq\alpha+i\) such that
		\[\forall \delta(\alpha  \prec\gamma\preceq \delta \prec \beta\Rightarrow\vDash^\delta \mathrm{B}(x)).\] So for any  \(x\),  \(\vDash^\beta \Visible\Bendcode{\mathrm{B}(x)}\), and consequently \(\vDash^\beta \forall x\Visible\Bendcode{\mathrm{B}(x)}\). 
		
		\item Otherwise, \( \vDash^\beta\Visible\Bendcode{\forall x\mathrm{B}(x)}\), and we would have \(\vDash^\delta \forall x \mathrm{B}(x)\) as from some  ordinal \(\delta\) below \(\beta \) and above \(\alpha\).  Let \(\delta=\alpha+(n+1)\), for  finite ordinal \(n\succeq 0\), be such an ordinal.  By instantiation, \(\vDash^\delta \mathrm{B(\mathrm{t}_{\overline{n}})}\),  this  entails that  \(\vDash^{\alpha+(n+1)} \mathrm{B}(t_{\overline{n}})\). As \(\vDash^{\alpha+(n+1)}\exists w(\langle w,\mathrm{t}_{\overline{n}}\rangle\in \mathrm{t}_{\omega})\), it follows that  \(\vDash^{\alpha +(n+1)}  \mathfrak{r}\notin \mathrm{t}_{\overline{n}}\). As a consequence, \(\vDash^{\alpha+1}  \mathfrak{r}\notin \mathrm{t}_{\overline{0}}\). But the latter entails  \(\vDash^\alpha (\mathfrak{r}\neq \mathfrak{r}\vee \mathfrak{r}\in\mathfrak{r}\vee \mathcal{T}\mathfrak{r}\in\mathfrak{r})\) which contradicts Lemma \ref{russelatlimit}, as \(\alpha\) is presupposed to be a limit ordinal. \qedhere \end{enumerate}\end{proof}

\begin{thm}\label{Barcannotmaxim}	 \[\not\vdash\Visible\Bendcode{\forall x\Visible\Bendcode{\mathrm{B}(x)}\to \Visible\Bendcode{\forall x\mathrm{B}(x)}}. \]
\end{thm}

\begin{proof} It suffices to appeal to Lemma \ref{failureofthebarcanformula}. 
\end{proof}

\subsection{Omega-consistency as confinement}\label{omegaconsistency}

\parencite[]{McGee1985}\index{McGee, Vann Roger}  isolated a  theory of truth which is consistent but $\omega$-inconsistent. \parencite{FriedmanSheard1987}\index{Friedman, Harvey}\index{Sheard, Alec Michael,  III {\scriptsize aka} Michael Sheard} proposed a  more substantial theory of truth, the \emph{Friedman-Sheard theory}, which inherits that  $\omega$-inconsistency property, and \parencite{Halbach1994}\index{Halbach, Volker} found that its proof-theoretic strength is the same as  the theory of ramified analysis for all finite levels, viz. $\Gamma_0$. 

An essential ingredient  in the proof of McGee's negative result fails\index{McGee, Vann Roger} in $\mathbf{VT}$, viz.  the equation we transform as
 \begin{equation}\label{mcgeecondition}
      \forall x(x\in\omega\to\Visible\Bendcode{\mathrm{A}(x)})\to\Visible\Bendcode{\forall x(x\in \omega\to\mathrm{A}(x))}),
 \end{equation}

 to our present context. Equation \ref{mcgeecondition} is used at \parencite[399]{McGee1985}.\index{McGee, Vann Roger} 
 
 Notice that in \textbf{VT}, \[\vdash \forall x(x\in\omega\to\Visible\Bendcode{\mathrm{A}(x)})\leftrightarrow \forall x\Visible\Bendcode{x\in\omega\to\mathrm{A}(x)},\] as $\omega$ is orthodox. 
 
 A version of \ref{mcgeecondition} holds in McGee's contexts on account of his inclusion of
a true  version of the  Barcan-formula for truth. But the Barcan-formula for visibility fails \emph{visibly} in \textbf{VT}, in the sense of Theorem \ref{Barcannotmaxim}.

For these  reasons,  the McGee  argument for omega inconsistency cannot be carried through for \textbf{VT}. The price we pay for retaining $\omega$-consistency, however, is to give up the \emph{Visibilized}  version of Axiom \ref{ax:IA2}:~p~\pageref{ax:IA2}. 
 
\newpage

\section{Sets, propositions, truth, and paradox}\label{stepidentity}

\epigraph{
Visibility is the primitive notion;\\
step-identity distinguishes the sets;\\
proposition-sets that are visible, are true.
}{}

\noindent We open with a summary of the section. 
The step of a vision $a$ is
\[
\Anschauung{y}{y\in a}.
\]

\noindent $a$ is \emph{step-identical} iff $a$ is identical to its own step, so that $a=\Anschauung{y}{y\in a}$.

The notions of \emph{set} and \emph{truth} are introduced, on the basis of \emph{visions}, \emph{visibility}, and the step-identity condition just introduced. That a vision $v$ is a set is expressed as $\Setpredicate(v)$, and $\Setpredicate(v)$ is true iff v is a vision which satisfies the step-identity condition, so that $v=\Anschauung{x}{x\in v}$. We first show that set visions are stable in the revision process. Vision $v$ is \emph{minimal} iff it is of the form $\Anschauung{z}{A}$, where $z$ is the $\prec$-least variable\footnote{The well-founded relation $\prec$ on symbol strings is given by Definition \ref{stringset}} not free in $A$.
A \emph{proposition} is a minimal vision, and as the variable in its defining formula does not any longer play a role, we take $\Proposisjon{A}$ to be the proposition associated with the sentence $A$. 
Proposition $\Proposisjon{B}$ is true iff it is a visible set.

\subsection{Sets}

\begin{fact}\label{SI(a)nottheorem}
 No formula $A$ in \textbf{VL} ensures that $\Anschauung{z}{A} = \Anschauung{y}{y \in\Anschauung{z}{A}}$. 
\end{fact}

We write $\vDash A$ for $\Pi \beta$
\begin{thm}[(Set Pressure)]\leavevmode\par\vspace{2mm}\label{Stepidentitypressure}
\[
\Pi\beta(\vDash^{\beta}\Setpredicate(\Anschauung{z}{A})\Rightarrow \ \Bigl(\Pi \beta(\vDash^{\beta}\Visible\Bendcode{A}) \ \textsc{or} \  \Pi \beta(\vDash^{\beta}\Visible\Bendcode{\lnot A})\Bigr).
\] 
\end{thm}

\bigskip

\begin{proof}
Given the definition of $\mathsf{Set}$, we obviously have that 

\begin{equation}\label{SIto!(!AtoA)}
\Pi\beta\vDash^{\beta} \Setpredicate(\Anschauung{z}{A})\Rightarrow \Pi\beta\vDash^{\beta}\Visible\Bendcode{\Visible\Bendcode{A}\to A}.
\end{equation}
\medskip 

On account of Axiom \ref{ax:V6}, we have that

\begin{equation}\label{!(!AtoA)to(!Avee!not!}
\Pi\beta\vDash^{\beta}\Visible\Bendcode{\Visible\Bendcode{A}\to A}\to (\Visible\Bendcode{A}\vee\Visible\Bendcode{\lnot A}).
\end{equation}

By distribution of $\vDash^\beta$ over implication, 

\[
\Downarrow
\]
\begin{equation}\label{PiB!(!AtoA)toPiB(!Avee!not!} \Pi\beta\vDash^{\beta}\Visible\Bendcode{\Visible\Bendcode{A}\to A}\to \Pi\beta\vDash^{\beta}(\Visible\Bendcode{A}\vee\Visible\Bendcode{\lnot A}).
\end{equation}

\medskip

So by transitivity
\[
\Downarrow
\]

\begin{equation}\label{PiB!(!AtoA)toPiB(!Avee!not!_2} \Pi\beta\vDash^{\beta} \Setpredicate(\Anschauung{z}{A})\Rightarrow 
\Pi\beta\vDash^{\beta}(\Visible\Bendcode{A}\vee\Visible\Bendcode{\lnot A}).
\end{equation}

If $\Pi\beta\vDash^{\beta} \Setpredicate(\Anschauung{z}{A})$ and $\vDash^\varpi \Visible\Bendcode{A}$, then $\Sigma\gamma\Pi\delta(\gamma\preceq\delta\prec\varpi \Rightarrow \ \vDash^{\delta}A)$.  As 
\[
\Pi\gamma\vDash^{\gamma} \Setpredicate(\Anschauung{z}{A}),
\] it follows that even for an $\epsilon\prec\gamma, \vDash^\epsilon A$. The argument can be iterated down to $0$, so that
$\Pi \beta\Vdash^{\beta}\Visible\Bendcode{A}$. If it alternatively  holds that $\Pi\beta\vDash^{\beta} \Setpredicate(\Anschauung{z}{A})$ and $\vDash^\varpi \Visible\Bendcode{\lnot A}$, $\Pi \beta\Vdash^{\beta}\Visible\Bendcode{\lnot A}$.
\end{proof}

\medskip

\noindent

\subsection{Propositions: their notation and existence}

\begin{dfn}[of propositions]\leavevmode\par\vspace{2mm}

Given that $A$ is a sentence,

\[
\Anschauung{x}{A}
\]

is a proposition iff $x$ is the $\prec$-least variable which is not free in $A$.
\end{dfn}

\begin{dfn}[of notation for propositions]\leavevmode\par\vspace{2mm}

If $A$ is a sentence then $\langle\!\langle A\rangle\!\rangle$ denotes the proposition
\[
\Anschauung{x}{A},
\]
where $x$ is the $\prec$-least variable not free in $A$.
\end{dfn}

\subsection{Truth}

\begin{dfn}[of truth via visibility and sethod]\leavevmode\par\vspace{2mm}\label{dfnoftruth}

By Definition~\ref{bendcodesprimitivestrings}:~p.~\pageref{bendcodesprimitivestrings}, Bend coded sentences are visions. Consequently, they may be  arguments of both $\Visible$ and $\Setpredicate$:

\[
\TruthPred\Proposisjon{B} \DefEqual \Visible\Proposisjon{B} \wedge \Setpredicate\Proposisjon{B}.
\]
\end{dfn}

The definition extends uniformly to open propositions.

\begin{lem}\label{stepidentityglobality}
   For any sentence $B$. $\Pi \beta,\vDash^{\beta}\Setpredicate\Bendcode{B}$ or $\Pi \beta,\vDash^{\beta}\lnot\Setpredicate\Bendcode{B}$.
\end{lem}
\begin{proof}
This holds by identity theory.
\end{proof}
\begin{cor}
  For any sentence $B$. $\Pi \beta,\vDash^{\beta}\Proposisjon{B}$ or $\Pi \beta,\vDash^{\beta}\lnot\Proposisjon{B}$  
\end{cor}

\begin{thm}[on ordinal permanence for \texorpdfstring{$\TruthPred\Proposisjon{A}$}{TrA} and \texorpdfstring{$\lnot\TruthPred\Proposisjon{A}$}{notTrA}.]\label{truthpermanency}\leavevmode\par\vspace{2mm}

\begin{enumerate}[label=\thethm.\arabic*]
    \item\label{thm:Trpermanence} If for some ordinal $\gamma, \vDash^\gamma  \TruthPred\Proposisjon{A}$ then for \emph{any} ordinal
        $\beta, \vDash^\beta \TruthPred\Proposisjon{A}$. 
    \item\label{thm:notTrpermanence} If for some ordinal $\gamma, \vDash^\gamma  \lnot\TruthPred\Proposisjon{A}$ then for \emph{any} ordinal
        $\beta, \vDash^\beta \lnot\TruthPred\Proposisjon{A}$. 
\end{enumerate}
\end{thm}

\begin{proof}

The proof appeals to Theorem \ref{Stepidentitypressure} and Lemma \ref{stepidentityglobality}: 

\begin{enumerate}
    \item\ref{thm:Trpermanence} If for some ordinal 
$\gamma, \vDash^\gamma \TruthPred\Proposisjon{A}$, then $\vDash^{\gamma}\Visible\Proposisjon{A}$ and $\vDash^{\gamma}\Setpredicate\Proposisjon{A}$ for it. By Lemma \ref{stepidentityglobality}, $\Pi \beta\vDash^{\beta}\Setpredicate\Proposisjon{A}$. 
So by quantifier theory and the Set Pressure Theorem \ref{Stepidentitypressure},  $\forall \beta\vDash^\beta\Visible\Proposisjon{A}$. So by Definition \ref{dfnoftruth}, $\forall \beta\vDash^\beta\TruthPred\Proposisjon{A}.$

\item\ref{thm:notTrpermanence} If for some $\gamma, \vDash^\gamma \lnot\TruthPred\Proposisjon{B}$, then there cannot be another ordinal $\delta$ such that  
$\vDash^\delta\TruthPred\Proposisjon{B}$, as it would all conflict with Theorem \ref{thm:Trpermanence}. So it follows that $\Pi\beta\vDash^{\beta}\lnot\TruthPred\Proposisjon{B}$.
\end{enumerate}
\end{proof}

\subsection{Paradox}

Truth-theoretic paradoxes do not lead to contradiction in \textbf{VL}.
For the liar sentence $\Lambda$, one has
\[
\vdash \Lambda \leftrightarrow \lnot \TruthPred\Proposisjon{\Lambda}, \quad
\vdash \lnot \TruthPred\Proposisjon{\Lambda}, \quad
\vdash \lnot \TruthPred\Proposisjon{\lnot \Lambda}, \quad
\vdash \Lambda, \quad
\vdash \lnot \Setpredicate(\Lambda).
\]
No inconsistency follows, for the possible failure of set-hood for $\Lambda$ blocks the derivation of contradiction.

\subsection{The revenge paradox}

The sentence ``I am neither true nor step-identical'' is handled in a desirable manner.
Let
\[
\rho \leftrightarrow \lnot \TruthPred\Proposisjon{\rho} \wedge \lnot \Setpredicate(\rho).
\]

\noindent
Then one has
\[
\vdash \rho \wedge \lnot \TruthPred\Proposisjon{\rho} \wedge \lnot \Setpredicate(\rho).
\]

So the revenge liar just fails to be step-identical.

\subsection{The place for truth}

\noindent Recall Definition \ref{dfnoftruth}: 
\[
\vdash\TruthPred\Proposisjon{A}\Leftrightarrow (\Setpredicate{\Bendcode{A}}\wedge \Visible\Bendcode{A}).
\]

\begin{lem}[on Truth-Visibility Coincidence for Dicta]\leavevmode\par\vspace{2mm}

\[
\vdash\Setpredicate{\Bendcode{A}}\to(\TruthPred\Proposisjon{A}\leftrightarrow\Visible\Bendcode{A}).
\]    
\end{lem}

\begin{proof}
    By propositional logic, from Definition \ref{dfnoftruth}.
\end{proof}

\smallskip

\begin{cor}
In quantified form,

\[
\vdash\forall y\bigl(
\Setpredicate{\Bendcode{y}}
\to
(\TruthPred\Proposisjon{y}\leftrightarrow\Visible\Bendcode{y})
\bigr).
\]
\end{cor}

\medskip
\medskip

So for propositions, truth and visibility coincide. 

\noindent But \textbf{VT} by itself cannot derive even that $\vdash\TruthPred\Proposisjon{0=0}$, as it must be postulated axiomatically that certain propositions are sets. In doing the latter, caution must be exercised. But a minimal desideratum here should be that classical logic, the visible axioms of \textbf{VT}, and set-theoretical principles, which soon follow, be postulated axiomatically as step-identical.

\subsection{Expansion and reduction principles}

Definition \ref{dfnoftruth}: 
\[
\vdash\TruthPred\Proposisjon{A}\Leftrightarrow (\Setpredicate{\Bendcode{A}}\wedge \Visible\Bendcode{A}).
\]

But the Set Pressure Theorem, \ref{Stepidentitypressure}
\[
\Pi\beta(\vDash^{\beta}\Setpredicate(\Anschauung{z}{A})\Rightarrow \ \Bigl(\Pi \beta(\vDash^{\beta}\Visible\Bendcode{A}) \ \textsc{or} \  \Pi \beta(\vDash^{\beta}\Visible\Bendcode{\lnot A})\Bigr).
\]

Lemma \ref{stepidentityglobality}
   For any sentence $B$. $\Pi \beta,\vDash^{\beta}\Setpredicate\Bendcode{B}$ or $\Pi \beta,\vDash^{\beta}\lnot\Setpredicate\Bendcode{B}$

\subsection{Recovery of pretheoretically intuitive truth principles}
We have seen that strange things may happen when $\Setpredicate\Bendcode{A}$ is false,  and the sentence $A$ may in such cases lead to paradox. The following result, however, may be assuring:

\begin{thm}[of Recovery]\leavevmode\par\vspace{2mm}\label{stepidentityrecovery}  

 If \ $\vdash\Setpredicate\Proposisjon{A}$, then 
 
\[
\vdash A \leftrightarrow \TruthPred\Proposisjon{A},
\quad
\vdash \forall x\,\TruthPred\Proposisjon{A} \leftrightarrow \TruthPred\Proposisjon{\forall x A},
\quad
\vdash \TruthPred\Proposisjon{\exists x A} \leftrightarrow \exists x\,\TruthPred\Proposisjon{A}.
\]
\end{thm}

\begin{proof} of $\vdash \Setpredicate\Proposisjon{A}\wedge A \to \ \TruthPred\Proposisjon{A}$

\bigskip

    We argue semantically, and assume

    \begin{equation}
        \vDash^\varpi (\Setpredicate\Proposisjon{A}\wedge A \wedge \ \lnot\TruthPred\Proposisjon{A}).
        \end{equation}

By Theorem \ref{stepidentityglobality} and
 Definition \ref{dfnoftruth} of truth:
 
\begin{equation}\label{nexttolastinstepidentityrecoveryproof}
\Pi\beta\Pi\vDash^\beta \Setpredicate\Proposisjon{A} \ \& \ 
        \vDash^\varpi A \wedge \lnot\Visible\Bendcode{A}.
        \end{equation}

By the Set Pressure Theorem \ref{stepidentityglobality} and $\Pi\beta\Pi\vDash^\beta \Setpredicate\Proposisjon{A}$: 

\begin{equation}\label{kontradiksjonfonektingavhalveT}
    \Pi \beta(\vDash^{\beta}\Visible\Bendcode{A}) \ \textsc{or} \  \Pi \beta(\vDash^{\beta}\Visible\Bendcode{\lnot A})
\end{equation}

The first disjunct must fail, as by Equation \ref{nexttolastinstepidentityrecoveryproof}, $\vDash^\varpi \lnot\Visible{\Bendcode{A}}$. The second conjunct fails for the same reason,
as it entails $\vDash^\omega\Visible\Bendcode{\lnot A})$, and therefore $\vDash^\varpi \lnot\Visible{\Bendcode{A}}$ given the validity of $\vDash^\omega\Visible\Bendcode{\lnot A}\to\lnot\Visible\Bendcode{A}$.

\end{proof}

We next prove the other direction:

\begin{proof} of $\vdash \TruthPred\Proposisjon{A} \to \Setpredicate\Proposisjon{A}\wedge A$.

\bigskip

    We argue semantically, and assume

    \begin{equation}\label{assumptionfortruthAtoSIbendofAandA}
        \vDash^\varpi \TruthPred\Proposisjon{A}\wedge\lnot(\Setpredicate\Proposisjon{A}\wedge A).
        \end{equation}

\text{By logic:}

 \begin{equation}\label{confitionalformofVTtruthdefinition}
        \vDash^\varpi \TruthPred\Proposisjon{A}\wedge(\Setpredicate\Proposisjon{A}\to \lnot A).
        \end{equation}

By Definition \ref{dfnoftruth}:

\begin{equation}\label{fromtruthtostepidentity}
    \vDash^\varpi \TruthPred\Proposisjon{A}\to\Setpredicate\Proposisjon{A},
\end{equation}

From Equation \ref{confitionalformofVTtruthdefinition} and Equation \ref{fromtruthtostepidentity}:

\begin{equation}\label{TrAtonotA}
    \vDash^\varpi \TruthPred\Proposisjon{A}\to \lnot A).
\end{equation}

From Definition \ref{dfnoftruth}:
    
\begin{equation}\label{truebendofAtovisiblebedofA}
    \vDash^\varpi \TruthPred\Proposisjon{A}\to \Visible\Bendcode{A}).
\end{equation}

From Axiom \ref{ax:IA1} and Equation \ref{TrAtonotA}:

\begin{equation}\label{fromTrBendofAtonotvisibleBendofA}
     \vDash^\varpi \TruthPred\Proposisjon{A}\to \lnot\Visible\Bendcode{A}).
\end{equation}

From Equations \ref{assumptionfortruthAtoSIbendofAandA}, \ref{truebendofAtovisiblebedofA}, \ref{fromTrBendofAtonotvisibleBendofA} and logic:

\begin{equation}\label{FromTrBendofAtonotVisibleBendofA}
    \vDash^\varpi \TruthPred\Proposisjon{A}\to \Visible\Bendcode{A}\wedge\lnot\Visible\Bendcode{A}).
\end{equation}

By Equations \ref{assumptionfortruthAtoSIbendofAandA}, \ref{assumptionfortruthAtoSIbendofAandA}, \ref{FromTrBendofAtonotVisibleBendofA} and logic:

\begin{equation}
    \vDash^\varpi \Visible\Bendcode{A}\wedge\lnot\Visible\Bendcode{A}
\end{equation}
\end{proof}

\subsection{Conclusion}

The present framework remains intentionally weak. Set membership is governed
only by the semantic clause above, and no global closure principles are assumed.
Stronger set-theoretic principles may be obtained by postulating
\emph{manifestation points}. The author introduced these in
\parencite[345]{Bjordal2012}. In \parencite[87--105]{Bjordal2025arxiv}
it was shown that the librationist theory may be extended, by means of
extra assumptions on the extension of the set
$\{x\mid \{y\mid A\}=\{z\mid z\in \{y\mid A\}\}\}$,
in such a way that it interprets Tarski--Grothendieck set theory.
Further research may extend and sharpen such results.

Classical logic is preserved throughout in \textbf{VL}. Visibility is the only semantic notion subject to specific regulation by the revision semantics. Truth introduces no independent nonclassical behavior: in non-paradoxical contexts it behaves classically, and in paradoxical contexts its restriction is a consequence of the rigidity of Set.

\newpage

  \section{Visionary set theory: a beginning}

  \subsection{The set-theoretic vision}

The ontology of \textbf{VT} consists of visions. Set theory is therefore not concerned with all objects whatsoever, but rather with those visions satisfying the sethood condition.

\begin{dfn}[of the set-theoretic vision]
\leavevmode\par\vspace{2mm}

Let

\[
\mathcal V_{\mathrm{Set}}
\coloneqq
\Anschauung{x}{\Setpredicate(x)}.
\]

\EndDef
\end{dfn}

We assume, as is reasonable, I think, that $\mathcal{V}_{\mathrm{Set}}$ is not paradoxical. So we posit

\[
\vdash a\in\mathcal V_{\mathrm{Set}}
\leftrightarrow
\Setpredicate(a).
\]

Thus the members of \(\mathcal V_{\mathrm{Set}}\) are precisely the sets.

This permits a natural formulation of set theory within \textbf{VT}. Rather than quantifying over all visions, set-theoretic quantification may be understood as quantification restricted to the members of \(\mathcal V_{\mathrm{Set}}\). Accordingly, one may introduce the abbreviations

\[
\forall_{\mathrm{Set}}xA
\coloneqq
\forall x\bigl(x\in\mathcal V_{\mathrm{Set}}\to A\bigr),
\]

and

\[
\exists_{\mathrm{Set}}xA
\coloneqq
\exists x\bigl(x\in\mathcal V_{\mathrm{Set}}\wedge A\bigr).
\]

The distinction between visions and sets is essential. Every set is a vision, but some visions are not sets. Consequently, the existence of the set-theoretic vision does not by itself imply that the set-theoretic vision is a set. Indeed, if sets are understood as hereditary sets, there is no reason to expect that

\[
\Setpredicate(\mathcal V_{\mathrm{Set}}).
\]

\subsection{Set quantification}

Since $\Setpredicate(a)$ entails $a=\Anschauung{y}{y\in a}$, statements involving set variables can be expanded by adjoining the corresponding step-identity condition.

The truth-theoretic quantifier-shift principles are not postulated directly. Rather, they are obtained from corresponding visibility-theoretic principles. Since truth in \textbf{VT} is defined through visibility and set-hood, the proper place to state such principles is at the level of visibility. The exception is that the Barcan-formula for visibility must be postulated as true, for it is not available as a visible principle in primitive visibility theory. Accordingly, we have

\begin{equation}\label{Setfirstquantifiershift}
\Setpredicate\Proposisjon{\exists_{\mathrm{Set}}x\Visible\Proposisjon{A(x)}
\to
\Visible\Proposisjon{\exists_{\mathrm{Set}}x\Proposisjon{A(x)}}}
\end{equation}

\begin{equation}\label{Setsecondquantifiershift}
\Setpredicate\Proposisjon{\Visible\Proposisjon{\exists_{\mathrm{Set}}x\,A(x)}
\to
\exists_{\mathrm{Set}}x\Visible\Proposisjon{A(x)}}
\end{equation}

\begin{equation}\label{Setthirdquantifiershift}
\Setpredicate\Proposisjon{\Visible\Proposisjon{\forall_{\mathrm{Set}}xA(x)}
\to
\forall_{\mathrm{Set}}x\Visible\Proposisjon{A(x)}}
\end{equation}

\begin{equation}\label{Setfourthquantifiershift}
\TruthPred\Proposisjon{\forall_{\mathrm{Set}}x\Visible\Proposisjon{A(x)}
\to
\Visible\Proposisjon{\forall_{\mathrm{Set}}xA(x)}}
\end{equation}

By the step-identicality of sets, \ref{Setfirstquantifiershift}--\ref{Setfourthquantifiershift} postulate step-identicality. The Barcan formula is expected to hold when quantification is restricted to \(\mathcal V_{\mathrm{Set}}\), since paradoxical sentences are not expected to be sets. By Theorem \ref{stepidentityrecovery}
on Recovery, the corresponding principles for truth follow automatically. In this way, truth inherits its quantificational behaviour from visibility together with the step-identity imposed by set-hood.

\newpage

\section*{Appendix: Hilbert on Mathematical Knowledge}
\addcontentsline{toc}{section}{Appendix: Hilbert on Mathematical Knowledge}


The following passage is taken from Hilbert's 1930 Königsberg lecture
``Naturerkennen und Logik'' \parencite[963]{Hilbert1930}.
The quotation is historically significant for the present discussion because it vividly expresses a strongly truth-directed conception of mathematical inquiry and proof.

\begin{quote}

``Wer die Wahrheit der großzügigen Denkweise und Weltanschauung, die aus diesen Worten Jacobis hervorleuchtet, empfindet, der verfällt nicht schrittlicher und unfruchtbarer Zweifelsucht; der wird nicht denen 
glauben, die heute mit philosophischer Miene und überlegenem Tone den gang prophezeien und sich in dem Ignorabimus gefallen. Für den Mathematiker gibt es kein Ignorabimus, und meiner Meinung nach auch für die wissenschaft überhaupt nicht. Einst sagte der Philosoph Comte — in der sicht, ein gewiß unlösbares Problem zu nennen —, daß es der Wissenschaft nie gelingen würde, das Geheimnis der chemischen Zusammensetzung der Himmelskörper zu ergründen. Wenige Jahre später wurde durch die analyse von Kirchhoff und Bunsen dieses Problem gelöst, und heute können wir sagen, daß wir die entferntesten
Sterne als wichtigste physikalische und chemische Laboratorien in Anspruch nehmen, wie wir solche auf der Erde gar nicht finden. Der wahre Grund, warum es Comte nicht gelang, ein unlösbares Problem zu finden, besteht meiner Meinung nach darin, daß es ein unlösbares Problem überhaupt nicht gibt. Statt des törichten Ignorabimus heiße im Gegenteil unsere Lösung: Wir müssen wissen. Wir werden wissen.''

\end{quote}

\clearpage

\printbibliography

\end{document}